\documentclass[10pt]{amsart}
\usepackage{amssymb}
\usepackage{amsmath,amssymb,amsfonts,amsthm,
latexsym, amscd, amsfonts, epsfig, epic}
\usepackage{mathrsfs}
\usepackage{color}
\usepackage{eucal}%    caligraphic-euler fonts: \mathcal{ }
\usepackage{eufrak}%   frak-euler        fonts: \mathfrak{ }
\usepackage{xspace}
\usepackage{a4wide}
\usepackage{colordvi}
\usepackage{eucal}
\newcounter{mycount}
\theoremstyle{plain}
\newtheorem{theorem}[mycount]{Theorem}
\newtheorem{corollary}[mycount]{Corollary}

\newtheorem{lemma}[mycount]{Lemma}
\newtheorem{proposition}[mycount]{Proposition}
\newtheorem{remark}{Remark}
\theoremstyle{definition}
\newtheorem{definition}{Definition}
\numberwithin{mycount}{section}

\theoremstyle{example}
\newtheorem{example}{Example}

\theoremstyle{remark}

\numberwithin{equation}{section} \numberwithin{figure}{section}
\numberwithin{definition}{section} \numberwithin{remark}{section}
\numberwithin{example}{section}
\begin{document}
\title[Outside nested decompositions and Schur function determinants]{Outside nested decompositions of skew diagrams and Schur function determinants}
\author{Emma Yu Jin}
\thanks{Corresponding author email: yu.jin@tuwien.ac.at. The author was supported by the German Research Foundation DFG, JI 207/1-1, and is supported by the Austrian Research Fund FWF, project SFB Algorithmic and Enumerative Combinatorics F50-02/03}
\email{yu.jin@tuwien.ac.at}
\address{Institut f\"ur Diskrete Mathematik und Geometrie, TU Wien, Wiedner Hauptstr. 8--10, 1040 Vienna, Austria}

\maketitle

\begin{abstract}
In this paper we describe the {\em thickened strips} and the {\em outside nested decompositions} of any skew shape $\lambda/\mu$. For any such decomposition $\Phi=(\Theta_1,\Theta_2,\ldots,\Theta_g)$ of the skew shape $\lambda/\mu$ where $\Theta_i$ is a thickened strip for every $i$, if $r$ is the number of boxes that are contained in any two distinct thickened strips of $\Phi$,
we establish a determinantal formula of the function $s_{\lambda/\mu}(X)p_{1^r}(X)$ with the Schur functions of thickened strips as entries, where $s_{\lambda/\mu}(X)$ is the Schur function of the skew shape $\lambda/\mu$ and $p_{1^r}(X)$ is the power sum symmetric function index by the partition $(1^r)$. This generalizes Hamel and Goulden's theorem on the outside decompositions of the skew shape $\lambda/\mu$ ({\em Planar decompositions of tableaux and Schur function determinants}, Europ. J. Combinatorics, 16, 461-477, 1995).
As an application of our theorem, we derive the number of $m$-strip tableaux which was first counted by Baryshnikov and Romik ({\em Enumeration formulas for Young tableaux in a diagonal strip}, ‎Israel J. Math, 178, 157-186, 2010)
via extending the transfer operator approach due to Elkies.
\end{abstract}
%%%%%%%%%%%%%%%%%%%%%%%%%%%%%%%%%%%%%%%%%%%
\tableofcontents
%%%%%%%%%%%%%%%%%%%%%%%%%%%%%%%%%%%%%%%%%%%
\section{Introduction and main results}
One of the most fundamental results on the symmetric functions is the determinantal expression of the Schur function $s_{\lambda/\mu}(X)$ for any skew shape $\lambda/\mu$; see \cite{Macdonald,Stanley:ec2}. The Jacobi-Trudi determinant \cite{Jacobi,Gessel-Viennot} and its dual \cite{Macdonald,Gessel-Viennot}, the Giambelli determinant \cite{Giambelli,Stembridge} as well as Lascoux and Pragacz's rim ribbon determinant \cite{LP,Ueno} are all of this kind. Hamel and Goulden \cite{HG} remarkably found that all above mentioned determinants for the Schur function $s_{\lambda/\mu}(X)$ can be unified through the concept of {\em outside decompositions} of the skew shape $\lambda/\mu$.

In what follows all definitions will be postponed until subsection~\ref{ss:2} and we first present Hamel and Goulden's theorem (Theorem~\ref{T:HG}).
%%%%%%%%%%%%%%%%%%%%%%%%%%%%%%%%%%%%%%%%%%
\begin{theorem}[\cite{HG}]\label{T:HG}
If the skew diagram of $\lambda/\mu$ is edgewise connected. Then, for any outside decomposition $\phi=(\theta_1,\theta_2,\ldots,\theta_g)$ of the skew shape $\lambda/\mu$, it holds that
\begin{align}\label{E:HG}
s_{\lambda/\mu}(X)=
\det[s_{\theta_i\#\theta_j}(X)]_{i,j=1}^{g},
\end{align}
where $s_{\varnothing}(X)=1$ and $s_{\theta_i\#\theta_j}(X)=0$ if $\theta_i\#\theta_j$ is undefined.
\end{theorem}
%%%%%%%%%%%%%%%%%%%%%%%%%%%%%%%%%%%%%%%%%
Their proof is based on a lattice path construction and the Lindstr\"{o}m-Gessel-Viennot methodology \cite{Gessel-Viennot,Stembridge}. In this paper we generalize the concept of outside decompositions even further.
%%%%%%%%%%%%%%%%%%%%%%%%%%%%%%%%%%%%%%%%%
\subsection{Our main results}
We introduce the concept of {\em outside nested decompositions} of the skew shape $\lambda/\mu$ and our first main result is a generalization of Theorem~\ref{T:HG} with respect to any outside nested decomposition $\Phi=(\Theta_1,\Theta_2,\ldots,\Theta_g)$ of the skew shape $\lambda/\mu$.

For any such decomposition $\Phi=(\Theta_1,\Theta_2,\ldots,\Theta_g)$ of the skew shape $\lambda/\mu$ where $\Theta_i$ is a thickened strip for every $i$, if $r$ is the number of boxes that are contained in two distinct thickened strips of $\Phi$. Then, our main theorem provides a determinantal formula of the function $s_{\lambda/\mu}(X)p_{1^r}(X)$ with the Schur functions of thickened strips as entries. The precise statement is the following.
%%%%%%%%%%%%%%%%%%%%%%%%%%%%%%
\begin{theorem}\label{T:thick}
If the skew diagram of $\lambda/\mu$ is edgewise connected. Then, for any outside nested decomposition $\Phi=(\Theta_1,\Theta_2,\ldots,\Theta_g)$ of the skew shape $\lambda/\mu$, we set
$r=\sum_{i=1}^g \vert\Theta_i\vert-\vert\lambda/\mu\vert$ that is the number of common special corners of $\Phi$ and we have
\begin{align}\label{E:HG}
p_{1^{r}}(X)\,s_{\lambda/\mu}(X)=
\det[s_{\Theta_i\#\Theta_j}(X)]_{i,j=1}^{g}\quad \mbox{ where }\, p_{1^{r}}(X)=(\sum_{i=1}^{\infty}x_i)^r,
\end{align}
$s_{\varnothing}(X)=1$ and $s_{\Theta_i\#\Theta_j}(X)=0$ if $\Theta_i\#\Theta_j$ is undefined. The function $p_{1^r}(X)$ is the power sum symmetric function index by the partition $(1^r)$ and $p_{1^r}(X)=1$ if $r=0$.
\end{theorem}
%%%%%%%%%%%%%%%%%%%%%%%%%%%%%%
When $r=0$ and all thickened strips $\Theta_i$ are strips, we retrieve Hamel and Goulden's theorem on the {\em outside decompositions} of the skew shape $\lambda/\mu$. With the help of Theorem~\ref{T:thick}, it suffices to find an outside nested decomposition with minimal number of thickened strips in order to reduce the order of the determinantal expression of the Schur function $s_{\lambda/\mu}(X)$.

%%%%%%%%%%%%%%%%%%%%%%%%%%%%%%%%%%%%%%%%%%%%%%%%%%%%%%%%%%%%%%%%%%%%
Let $\vert\lambda/\mu\vert$ and $f^{\lambda/\mu}$ denote the number of boxes contained in the skew shape $\lambda/\mu$ and the number of standard Young tableaux of shape $\lambda/\mu$ with the entries from $1$ to $\vert\lambda/\mu\vert$ (similarly for $\vert\Theta_i\#\Theta_j\vert$ and $f^{\Theta_i\#\Theta_j}$). Then, by applying the exponential specialization on both sides of (\ref{E:HG}), one immediately gets
\begin{corollary}\label{C:thick2}
If the skew diagram of $\lambda/\mu$ is edgewise connected. Then, for any outside nested decomposition $\Phi=(\Theta_1,\Theta_2,\ldots,\Theta_g)$ of the skew shape $\lambda/\mu$, we have
\begin{align}\label{E:HG1}
\,f^{\lambda/\mu}=
\vert\lambda/\mu\vert!\,\det\left[(a_{i,j}!)^{-1}f^{\Theta_i\#\Theta_j}\right]_{i,j=1}^{g}
\quad \mbox{ where }\,
a_{i,j}=\vert\Theta_i\#\Theta_j\vert,
\end{align}
$f^{\varnothing}=1$ and $f^{\Theta_i\#\Theta_j}=0$ if $\Theta_i\#\Theta_j$ is undefined.
\end{corollary}
It should be noted that the parameter $r$ vanishes in (\ref{E:HG1}). Our second main result is an enumeration of $m$-strip tableaux by applying Corollary~\ref{C:thick2}, which provides another proof of Baryshnikov and Romik's results in \cite{BR:07}. Baryshnikov and Romik \cite{BR:07} counted $m$-strip tableaux via extending the transfer operator approach due to Elkies \cite{Elkies}.
\subsection{Paper outline}
In subsection~\ref{ss:1} and \ref{ss:2} we introduce all necessary notations and definitions. In Section~\ref{S:proof1} we prove Theorem~\ref{T:thick} and Corollary~\ref{C:thick2}. In Section~\ref{E:app} we introduce the notion of $m$-strip tableaux and count the number of $m$-strip tableaux.
\subsection{Partitions and symmetric functions}\label{ss:1}
\begin{itemize}
\item A {\it partition} $\lambda$ of $n$, denoted by $\lambda\vdash n$, is a sequence $\lambda=(\lambda_1,\lambda_2,\ldots,\lambda_m)$ of non-negative integers such that $\lambda_1\ge\lambda_2\ge \cdots\ge\lambda_m\ge 0$ and their sum is $n$. The non-zero $\lambda_i$ are called the {\em parts} of $\lambda$ and the number of parts is the length of $\lambda$, denoted by $\ell(\lambda)$. Let $m_i=m_i(\lambda)$ denote the number of parts of $\lambda$ that equal $i$, we simply write $\lambda=(1^{m_1}2^{m_2}\cdots)$.
\item Given a partition $\lambda=(\lambda_1,\lambda_2,\ldots,\lambda_m)$, the {\it standard diagram} of $\lambda$ is a left-justified array of $\lambda_1+\lambda_2+\cdots+\lambda_m$ boxes with $\lambda_1$ in the first row, $\lambda_2$ in the second row, and so on.
\item A {\it skew diagram} of $\lambda/\mu$ (also called {\em a skew shape $\lambda/\mu$}) is the difference of two skew diagrams where $\mu\subseteq\lambda$. Note that the standard shape $\lambda$ is just the skew shape $\lambda/\mu$ when $\mu=\varnothing$.
\item The {\em content} of a box $\alpha$ in a skew shape $\lambda/\mu$ equals $t-s$ if the box $\alpha$ is in column $t$ from the left and row $s$ from the top of the skew shape $\lambda/\mu$. We refer to box $\alpha$ as box $(s,t)$ and $(s,t)$ is called its {\em coordinate}. A {\em diagonal of content $c$} in a skew diagram is a set of boxes with content $c$ in a skew diagram.
\item A skew diagram `starts' at a box (called {\em the starting box}) if that box is the bottommost and leftmost box in the skew diagram, and a skew diagram `ends' at a box (called {\em the ending box}) if that box is the topmost and rightmost box in the skew diagram.
\item A {\em semistandard Young tableau} (resp. {\em standard Young tableau}) of skew shape $\lambda/\mu$ is a filling of the boxes of the skew diagram of $\lambda/\mu$ with positive integers such that the entries strictly increase down each column and weakly (resp. strictly) increase left to right across each row.
\end{itemize}
In a semistandard Young tableau $T$ we use $T(\alpha)$ to represent the positive integer in the box $\alpha$ of $T$. The {\em Schur function}, $s_{\lambda/\mu}(X)$, in the variables $X=(x_1,x_2,\ldots)$, is given by
\begin{align*}
s_{\lambda/\mu}(X)=\sum_{T}\prod_{\alpha\in\lambda/\mu}x_{T(\alpha)},
\end{align*}
where the summation is over all semistandard Young tableaux $T$ of shape $\lambda/\mu$ and $\alpha\in \lambda/\mu$ means that $\alpha$ ranges over all boxes in the skew diagram of $\lambda/\mu$. In particular, $s_{\varnothing}(X)=1$. The {\em complete symmetric functions} $h_k(X)$ are defined by
\begin{align*}
h_k(X)=\sum_{1\le i_1\le \cdots \le i_k}\prod_{j=i_1}^{i_k}x_{j}\,\,\mbox{ if }\, k\ge 1,\,\, h_0(X)=1 \,\mbox{ and }\,h_k(X)=0 \,\mbox{ if }\, k<0.
\end{align*}
The Jacobi-Trudi identity is a determinantal expression of Schur function $s_{\lambda/\mu}(X)$ in terms of complete symmetric functions $h_k(X)$; see \cite{Macdonald,Stanley:ec2}.
\begin{theorem}[Jacobi-Trudi identity \cite{Jacobi}]\label{T:Jacobi-Trudi}
Let $\lambda/\mu$ be a skew shape partition, let $\lambda=(\lambda_1,\ldots,\lambda_k)$ and $\mu=(\mu_1,\ldots,\mu_k)$ have at most $k$ parts. Then
\begin{align*}
s_{\lambda/\mu}(X)=\det[h_{\lambda_i-\mu_j-i+j}(X)]_{i,j=1}^{k}.
\end{align*}
\end{theorem}
The classical Aitken formula for the number of standard Young tableaux of skew shape can be directly obtained by applying the exponential specialization on the Jacobi-Trudi identity; see Chapter 7 of \cite{Stanley:ec2}. We denote by $\vert\lambda/\mu\vert$ the number of boxes contained in the skew diagram of $\lambda/\mu$ and denote by $f^{\lambda/\mu}$ the number of standard Young tableaux of shape $\lambda/\mu$ with the entries from $1$ to $\vert\lambda/\mu\vert$.
%%%%%%%%%%%%%%%%%%%%%%%%%%%%%%%%%%%%%%%%%%%%
\begin{corollary}[Aitken formula]
Let $\lambda/\mu$ be a skew shape partition, let $\lambda=(\lambda_1,\ldots,\lambda_k)$ and $\mu=(\mu_1,\ldots,\mu_k)$ have at most $k$ parts. Then
\begin{align}\label{E:Aitken}
f^{\lambda/\mu}=\vert\lambda/\mu\vert!\det\left[\frac{1}{(\lambda_i-\mu_j-i+j)!}\right]_{i,j=1}^{k}.
\end{align}
\end{corollary}
It is clear that the order of the determinant in the Jacobi-Trudi identity and in the Aitken formula equals the number $\ell(\lambda)$ of parts in $\lambda$. Using (\ref{E:Aitken}) to compute $f^{\lambda/\mu}$ becomes difficult when the partitions $\lambda$ and $\mu$ are large, even when their difference $\lambda/\mu$ is small.
\subsection{Outside nested decompositions}\label{ss:2}
We start with the strips and outside decompositions. Hamel and Goulden described the notion of an {\em outside decomposition} of the skew shape $\lambda/\mu$, which generalizes Lascoux and Pragacz's rim ribbon decomposition \cite{LP}. With the help of Hamel and Goulden's theorem \cite{HG}, for any skew shape $\lambda/\mu$, one can reduce the order of the determinant in the Jacobi-Trudi identity to the number of strips contained in any outside decomposition of skew shape $\lambda/\mu$.

Two boxes are said to be edgewise connected if they share a common edge. A skew diagram $\theta$ is said to be {\em edgewise connected} if $\theta$ is an edgewise connected set of boxes.
\begin{definition}[strip]\label{D:ribbon}
A skew diagram $\theta$ is a {\em strip} if $\theta$ is edgewise connected and it contains no $2\times 2$ blocks of boxes.
\end{definition}
\begin{remark}
The strips in Definition~\ref{D:ribbon} are called `border strips' by Macdonald \cite{Macdonald} and are called `ribbons' by Lascoux and Pragacz \cite{LP}. We adopt the name `strips' from \cite{HG}.
\end{remark}
%%%%%%%%%%%%
\begin{definition}[outside decomposition \cite{HG}]\label{D:podec}
Suppose that $\theta_1,\theta_2,\ldots,\theta_g$ are strips of a skew diagram of $\lambda/\mu$ and every strip has a starting box on the left or bottom perimeter of the diagram and an ending box on the right or top perimeter of the diagram. Then we say the totally ordered set $\phi=(\theta_1,\theta_2,\ldots,\theta_g)$ is an {\em outside decomposition} of $\lambda/\mu$ if the union of these strips is the skew diagram of $\lambda/\mu$ and every two strips $\theta_i,\theta_j$ in $\phi$ are disjoint, that is, $\theta_i$ and $\theta_j$ have no boxes in common.
\end{definition}
\begin{remark}
The rim ribbon decomposition of $\lambda/\mu$ introduced by Lascoux and Pragacz \cite{LP} is an outside decomposition with minimal number of strips; see \cite{Stanley:02} and \cite{YYZ:05}.
\end{remark}
\begin{example}
See Figure~\ref{F:pic-3} for an outside decomposition and two non-outside decompositions where all boxes are marked by black dots. The first two decompositions in Figure~\ref{F:pic-3} are not outside decompositions since the strip $\theta_1=(5,1)$ of the left one has a starting box neither on the left nor on the bottom perimeter of the skew diagram and the strip $\theta_2=(3)$ of the middle one has an ending box neither on the right nor on the top perimeter of the skew diagram.
\end{example}
\begin{figure}[htbp]
\begin{center}
\includegraphics[scale=0.8]{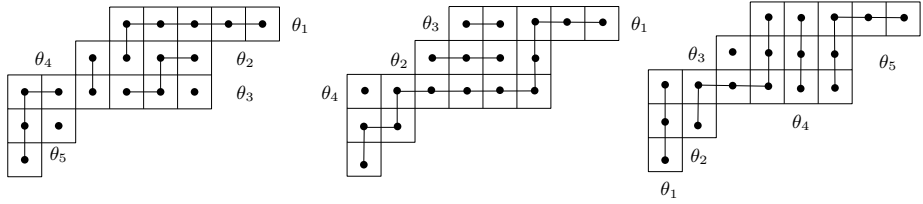}
\caption{Two non-outside decompositions (left and middle) and one outside decomposition (right) of skew shape $(8,6,6,2,1)/(3,2)$.
\label{F:pic-3}}
\end{center}
\end{figure}
We next introduce the notion of {\em thickened strips} and we will decompose the skew diagram of $\lambda/\mu$ into a sequence of thickened strips, in order to extend Hamel and Goulden's theorem \cite{HG} on the determinantal expression of the Schur function $s_{\lambda/\mu}(X)$. Our extension is motivated by the enumeration of $(2k+1)$-strip tableaux where any outside decomposition of $(2k+1)$-strip diagram with $n$ columns consists of at least $n$ strips (see Subsection~\ref{ss:odd}). So the order of the determinantal expression of $s_{\lambda/\mu}(X)$ can not be further reduced by applying Hamel and Goulden's theorem (Theorem~\ref{T:HG}).
%%%%%%%%%%%%%%%%%%%%%%%
\begin{definition}[thickened strip]\label{D:to}
A skew diagram $\Theta$ is a thickened strip if $\Theta$ is edgewise connected and it neither contains a $3\times 2$ block of boxes nor a $2\times 3$ block of boxes.
\end{definition}
\begin{remark}
By definition the only difference between strips and thickened strips is that thickened strips could have some $2\times 2$ blocks of boxes; see Figure~\ref{F:pic-9}.
\end{remark}
\begin{figure}[htbp]
\begin{center}
\includegraphics[scale=0.8]{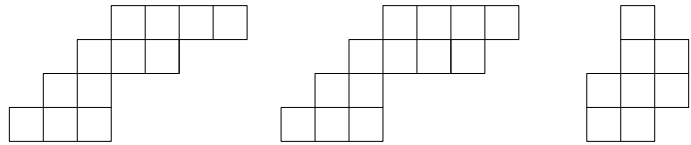}
\caption{The left one is a thickened strip, while the middle one and the right one are not thickened strips.
\label{F:pic-9}}
\end{center}
\end{figure}
We next define the {\em corners} and the {\em special corners} of a thickened strip $\Theta_i$ because in contrast to the outside decompositions, we allow two thickened strips in an outside nested decomposition to have special corners in common. In what follows, note that the box $(s,t)$ always refers to the box with coordinate $(s,t)$ in the skew diagram of $\lambda/\mu$.
\begin{definition}(corner, special corner)\label{D:SC}
When a thickened strip $\Theta_i$ has more than one box, we define that {\em a corner} $(s,t)$ of a thickened strip $\Theta_i$ is an upper corner or a lower corner, where {\em an upper corner} $(s,t)$ of $\Theta_i$ is a box $(s,t)$ such that neither the box $(s-1,t)$ nor the box $(s,t-1)$ is contained in $\Theta_i$. Likewise, {\em a lower corner} $(s,t)$ of $\Theta_i$ is a box $(s,t)$ such that neither the box $(s+1,t)$ nor the box $(s,t+1)$ is contained in $\Theta_i$. We say that a corner $(s,t)$ of a thickened strip $\Theta_i$ is {\em special} if the corner $(s,t)$ satisfies one of the following conditions:
\begin{enumerate}
\item the corner $(s,t)$ is the starting box or the ending box of $\Theta_i$;
\item the corner $(s,t)$ is contained in a $2\times 2$ block of boxes of $\Theta_i$.
\end{enumerate}
\end{definition}
\begin{example}
Consider the thickened strip in Figure~\ref{F:pic-9} (the left one), the only corner that is not special in this thickened strip is the box $(2,3)$.
\end{example}
Now we are ready to present the outside thickened strip decomposition.
\begin{definition}[outside thickened strip decomposition]\label{D:todec}
Suppose that $\Theta_1,\Theta_2,\ldots,\Theta_g$ are thickened strips in the skew diagram of $\lambda/\mu$ and every thickened strip has a starting box on the left or bottom perimeter of the diagram and an ending box on the right or top perimeter of the diagram.
Then we say the totally ordered set $\Phi=(\Theta_1,\Theta_2,\ldots,\Theta_g)$ is an {\em outside thickened strip decomposition} of the skew diagram of $\lambda/\mu$ if the union of the thickened strips $\Theta_i$ of $\Phi$ is the skew diagram of $\lambda/\mu$, and for all $i,j$, one of the following statements is true:
\begin{enumerate}
\item two thickened strips $\Theta_i$ and $\Theta_j$ are disjoint, that is, $\Theta_i$ and $\Theta_j$ have no boxes in common;
\item one thickened strips $\Theta_j$ is on the right side or the bottom side of the other thickened strip $\Theta_i$ and they have some special corners in common, where each common special corner $(s,t)$ is a lower corner of $\Theta_i$ and an upper corner of $\Theta_j$.
    %\begin{itemize}
    %\item box $(s,t)$ is a lower corner of $\Theta_i$ and an upper corner of $\Theta_j$;
    %\item box $(s,t)$ is not contained in any thickened strip of $\Phi$ other than $\Theta_i$ and $\Theta_j$.
    %\end{itemize}
\end{enumerate}
Every special corner of a thickened strip in $\Phi$ is called a {\em special corner of $\Phi$} and every common special corner of any two distinct thickened strips of $\Phi$ is called a {\em common special corner of $\Phi$}.
\end{definition}
\begin{remark}
If $\Theta_i$ has only one box $(s,t)$ and box $(s,t)$ is also a special corner of $\Theta_j$. Then the outside thickened strip decomposition $\Phi$ is essentially the same to the one without $\Theta_i$. So we exclude this scenario.
\end{remark}
\begin{example}
Figure~\ref{F:pic2-01} (middle, right) shows an outside thickened strip decomposition $\Phi=(\Theta_1,\Theta_2,\Theta_3)$ of the skew diagram of $(6,6,6,4)/(3,1)$ where the boxes $(4,1)$ and $(3,3)$ are the common special corners of $\Theta_2$ and $\Theta_3$. The box $(2,5)$ is the only common special corner of $\Theta_1$ and $\Theta_2$. In Figure~\ref{F:pic2-01} and Figure~\ref{F:pic-12} every common special corner of $\Phi$ is marked by a black square, while other boxes are marked by black dots.
\end{example}
%%%%%%%%%%%%%%%%%%%%%%%%%%%%%%%%%%%%
\begin{figure}[htbp]
\begin{center}
\includegraphics[scale=0.8]{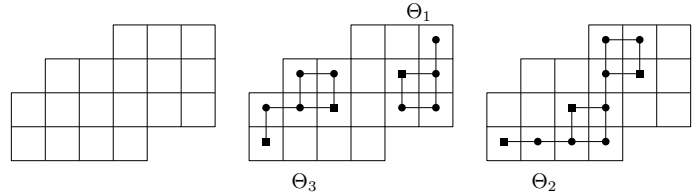}
\caption{An outside thickened strip decomposition $\Theta_1,\Theta_2,\Theta_3$ (middle, right) of the skew diagram $(6,6,6,4)/(3,1)$ (left).\label{F:pic2-01}}
\end{center}
\end{figure}
%%%%%%%%%%%%%%%%%%%%%%%%%%%%%%%%%%%%
\begin{example}
Figure~\ref{F:pic-12} (middle, right) gives an outside thickened strip decomposition $\Phi=(\Theta_1,\Theta_2,\Theta_3,\Theta_4,\Theta_5)$ of the skew diagram of $(8,8,8,7,4)/(3,1)$ where all the common special corners of $\Phi$ are boxes $(2,3),(2,5),(4,5),(3,6),(2,7),(1,8)$.
\end{example}
%%%%%%%%%%%%%%%%%%%%%%%%%%%%
\begin{figure}[htbp]
\begin{center}
\includegraphics[scale=0.8]{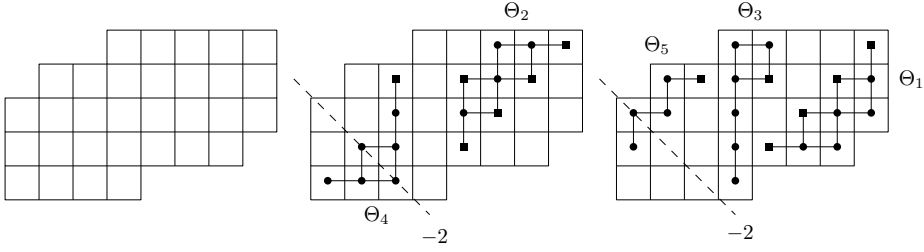}
\caption{An outside thickened strip decomposition of a skew diagram (left) where the thickened strips $\Theta_2,\Theta_4$ and $\Theta_1,\Theta_3,\Theta_5$ are drawn separately (middle, right). The dashed line with integer $-2$ represents the diagonal of content $-2$.\label{F:pic-12}}
\end{center}
\end{figure}
We observe that, unlike the strips in any outside decomposition, the thickened strips in any outside thickened strip decomposition $\Phi$ are {\em not necessarily nested}; see Definition~\ref{D:nesttodec}. However, the {\em nested property} of thickened strips in an outside thickened strip decomposition is of central importance in the proof of Theorem~\ref{T:thick}. In view of this, we need to introduce the {\em enriched diagrams} and the {\em directions} of all boxes in the skew shape $\lambda/\mu$ to describe the nested property of thickened strips.
%%%%%%%%%%%%%%%%%%%%%%%%%%%%
\begin{definition}[enriched diagram]\label{D:enrich}
Suppose that $\Phi=(\Theta_1,\Theta_2,\ldots,\Theta_g)$ is an outside thickened strip decomposition of the skew shape $\lambda/\mu$, for every $i$ such that $1\le i\le g$, and for box $(s,t)$ that is the starting box or the ending box of $\Theta_i$, we shall add new boxes to $\Theta_i$ according to the following rules:
\begin{enumerate}
\item if box $(s,t)$ is a lower corner of $\Theta_i$ and an upper corner of some other thickened strip in $\Phi$, we add boxes $(s,t-1),(s-1,t),(s-1,t-1)$ that are not contained in $\Theta_i$ to $\Theta_i$;
\item if box $(s,t)$ is an upper corner of $\Theta_i$ and a lower corner of some other thickened strip in $\Phi$, we add boxes $(s,t+1),(s+1,t),(s+1,t+1)$ that are not contained in $\Theta_i$ to $\Theta_i$.
\end{enumerate}
where all the coordinates of new boxes are {\em relative} to the coordinates of the boxes in the skew diagram of $\lambda/\mu$. We denote by $D(\Theta_i)$ the diagram after adding the new boxes to $\Theta_i$ and we call $D(\Theta_i)$ {\em an enriched thickened strip}. If neither the starting box nor the ending box of $\Theta_i$ satisfies $(1)$ or $(2)$, then $D(\Theta_i)=\Theta_i$. {\em An enriched diagram} $D(\Phi)$ is the union of all enriched thickened strips $D(\Theta_i)$ for every $\Theta_i$ of $\Phi$.
\end{definition}
\begin{example}
In Figure~\ref{F:pic2-01} the box $(4,1)$ contained in $\Theta_3$ and $\Theta_2$ is the only box that satisfies conditions $(1)$ and $(2)$ of Definition~\ref{D:enrich}. So we add the boxes $(4,0),(3,0)$ to $\Theta_3$ and add the boxes $(5,1),(5,2)$ to $\Theta_2$; see Figure~\ref{F:pic2-1} where all newly added boxes are colored grey.
\end{example}
%%%%%%%%%%%%%%%%%%%%%%%%%%%%%%%%%%%%%%%
\begin{figure}[htbp]
\begin{center}
\includegraphics[scale=0.8]{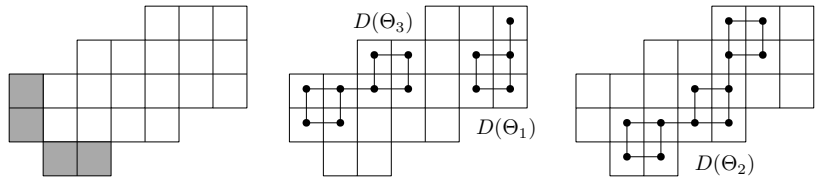}
\caption{The enriched diagram $D(\Phi)$ (left) and the enriched thickened strips (middle, right) where $\Phi=(\Theta_1,\Theta_2,\Theta_3)$ is given in Figure~\ref{F:pic2-01}.\label{F:pic2-1}}
\end{center}
\end{figure}
%%%%%%%%%%%%%%%%%%%%%%%
\begin{remark}
%Every two enriched thickened strips in $D(\Phi)$ are not necessarily edgewise disconnected; see Figure~\ref{F:pic-13} where the enriched thickened strips $D(\Theta_3)$ and $D(\Theta_4)$ are edgewise connected. Furthermore,
The enriched diagram $D(\Phi)$ may not be a skew diagram; see Figure~\ref{F:pic2-1} and \ref{F:pic-13}.
\end{remark}
%%%%%%%%%%%%%%%%%%%%%%%%%%
\begin{figure}[htbp]
\begin{center}
\includegraphics[scale=0.8]{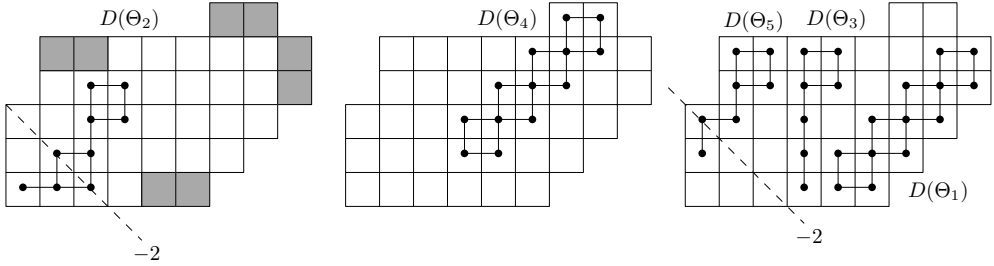}
\caption{An enriched diagram $D(\Phi)$ with enriched thickened strips $D(\Theta_i)$ where the outside thickened strip decomposition $\Phi$ is given in Figure~\ref{F:pic-12} and the dashed line with integer $-2$ represents the diagonal of content $-2$. \label{F:pic-13}}
\end{center}
\end{figure}
%%%%%%%%%%%%%%%%%%%%%%%%
With the help of enriched diagram $D(\Phi)$, one can define the directions of all boxes other than the special corners of $\Phi$ in the skew diagram of $\lambda/\mu$. For every box $(s,t)$ of the skew diagram of $\lambda/\mu$, if box $(s,t)$ is {\em not} a special corner of $\Phi$, then box $(s,t)$ is contained in only one thickened strip $\Theta_i$ of $\Phi$. We may define the {\em direction} of box $(s,t)$ in the enriched diagram $D(\Phi)$ according to the following rules:
\begin{enumerate}
\item if both boxes $(s-1,t)$ and $(s,t+1)$ are contained in the enriched thickened strip $D(\Theta_i)$ of $D(\Phi)$, then we say the box $(s,t)$ {\em goes right and up};
\item if not both boxes $(s-1,t)$ and $(s,t+1)$ are contained in $D(\Theta_i)$, then we say that the box $(s,t)$ {\em goes right} or {\em goes up} if $(s,t+1)$ or $(s-1,t)$ is contained in $D(\Theta_i)$;
\item if neither box $(s,t+1)$ nor box $(s-1,t)$ is contained in $D(\Theta_i)$, then box $(s,t)$ must be the ending box of $\Theta_i$, thus it must be on the top or right perimeter of the skew diagram of $\lambda/\mu$, and we say that box $(s,t)$ {\em goes up} if it is on the top perimeter of $\lambda/\mu$ and that box $(s,t)$ {\em goes right} if it is on the right perimeter but not on the top perimeter of the skew diagram of $\lambda/\mu$. %see Figure~\ref{F:pic-141}.
\end{enumerate}
%%%%%%%%%%%%%%%%%%%%%%%%%%%%%%%%%%%%%%%%
%If both boxes $(s-1,t)$ and $(s,t+1)$ are contained in the enriched thickened strip $D(\Theta_i)$ of $D(\Phi)$, then at least one of boxes $(s-1,t)$ and $(s,t+1)$ is a special corner of $\Phi$.
%\begin{figure}[htbp]
%\begin{center}
%\includegraphics[scale=0.8]{pic-141.eps}
%\caption{The box $(s,t)$ that goes right, or goes up, or goes right and up, or is special corners of $\Phi$ where box $(s,t)$ is marked by a black dot and the arrow represents the direction of box $(s,t)$.\label{F:pic-141}}
%\end{center}
%\end{figure}
%%%%%%%%%%%%%%%%%%%%%%%%%%%%%%%%%%%%%%%%
%The directions of all boxes that are not special corners allow us to define the {\em nested property} and the {\em outside nested decomposition} more precisely.
\begin{definition}[outside nested decomposition]\label{D:nesttodec}
An outside thickened strip decomposition $\Phi$ is an outside nested decomposition if $\Phi=(\Theta_1,\Theta_2,\ldots,\Theta_g)$ is {\em nested}, that is,
for all $c$, one of the following statements is true:
\begin{enumerate}
\item all boxes of content $c$ all go right or all go up;
\item all boxes of content $c$ or all boxes of content $(c+1)$ are all special corners of $\Phi$.
\end{enumerate}
\end{definition}
\begin{remark}
It should be noted that all boxes of content $(c+1)$ are special corners of $\Phi$ if and only if all boxes of content $c$ all go right and up. Definition~\ref{D:nesttodec} is analogous to the nested property of the strips in any outside decomposition where all boxes on the same diagonal of the skew shape $\lambda/\mu$ all go right or all go up; see \cite{CYY,HG}.
\end{remark}
\begin{example}
By Definition~\ref{D:nesttodec} the outside thickened strip decomposition in Figure~\ref{F:pic-12} is not an outside nested decomposition because two boxes on the diagonal of content $-2$ are special corners, but one box goes right, while the outside thickened strip decomposition in Figure~\ref{F:pic2-01} is an outside nested decomposition because all boxes on the diagonal of content $-3,0,3$ are all special corners, all boxes on the diagonal of content $1,4,5$ all go up, and all boxes on the diagonal of content $-2$ all go right.
\end{example}
%%%%%%%%%%%%%%%%%%%%%%%%%%%%%%%%%%%%%%%
Hamel and Goulden \cite{HG} defined a non-commutative operation $\#$ for every two strips of an outside decomposition $\phi=(\theta_1,\theta_2,\ldots,\theta_g)$ of the skew shape $\lambda/\mu$, also when the skew shape $\lambda/\mu$ is edgewise disconnected. Subsequently, Chen, Yan and Yang \cite{CYY} came up with the notion of cutting strips so as to derive a transformation theorem for Hamel and Goulden's determinantal formula, in which one of the key ingredients is a bijection between the outside decompositions of a given skew diagram and the cutting strips.

Based on these previous work, we will extend the non-commutative operation $\#$ for every two thickened strips of an outside nested decomposition $\Phi=(\Theta_1,\Theta_2,\ldots,\Theta_g)$ of the skew shape $\lambda/\mu$. In order to provide a simple definition of $\Theta_i\#\Theta_j$, we need to introduce the {\em thickened cutting strips}, which are called `cutting strips' for any outside decomposition in \cite{CYY}.
%%%%%%%%%%%%%%%%%%%%%%%%%%%%%%%%%%%
\begin{definition}[thickened cutting strips]\label{D:tcs}
The {\em thickened cutting strip} $H(\Phi)$ with respect to an outside nested decomposition $\Phi=(\Theta_1,\Theta_2,\ldots,\Theta_g)$ is a thickened strip obtained by successively superimposing the enriched thickened strips $D(\Theta_1),D(\Theta_2),\ldots,D(\Theta_g)$ of $D(\Phi)$ along the diagonals. %so that two boxes in a diagonal of $H(\Phi)$ are corners of $H(\Phi)$ if and only if all boxes in the same diagonal of $D(\Phi)$ are all special corners of $\Phi$, and the only box in a diagonal of $H(\Phi)$ goes right or/and up if and only if all boxes in the same diagonal of $D(\Phi)$ go right or/and up.
\end{definition}
%%%%%%%%%%%%%%%%%%%%%%%%%%%%%%%%%%%
We say that {\em a box $\alpha$ of the thickened cutting strip $H(\Phi)$ has content $c$} if box $\alpha$ is on the diagonal of content $c$ in the skew diagram of $\lambda/\mu$ and we represent each box of the thickened cutting strip $H(\Phi)$ as follows:
\begin{enumerate}
\item box $[c]$ denotes the unique box of $H(\Phi)$ with content $c$;
\item box $[c,+]$ and box $[c,-]$ denote the upper and the lower corner of $H(\Phi)$ with content $c$ if they are contained in a $2\times 2$ block of boxes in $H(\Phi)$.
\end{enumerate}
Because of the nested property in Definition~\ref{D:nesttodec}, the thickened cutting strip $H(\Phi)$ with respect to any outside nested decomposition $\Phi$ is a thickened strip. %and we define {\em the direction of box $[c]$ in $H(\Phi)$} is the uniform direction of all boxes of content $c$ in the skew diagram of $\lambda/\mu$, and boxes $[c,+],[c,-]$ denote all special corners of $\Phi$.
%%%%%%%%%%%%%%%%%%%%%%%%%%%%%%%%%%%
\begin{example}
Consider the outside nested decomposition $\Phi=(\Theta_1,\Theta_2,\Theta_3)$ in Figure~\ref{F:pic2-01}, the thickened cutting strip with respect to $\Phi$ is constructed in Figure~\ref{F:pic-15}, where the dashed lines represent the diagonals of content $-3,0,3$ respectively.
\end{example}
%%%%%%%%%%%%%%%%%%%%%%%%%
\begin{figure}[htbp]
\begin{center}
\includegraphics[scale=0.8]{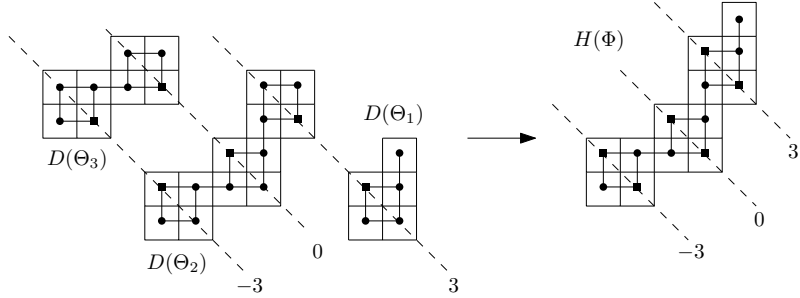}
\caption{The thickened cutting strip $H(\Phi)$ (right) with respect to the outside nested decomposition $(\Theta_1,\Theta_2,\Theta_3)$ in Figure~\ref{F:pic2-01} and the enriched thickened strips $D(\Theta_1),D(\Theta_2),D(\Theta_3)$ are given in Figure~\ref{F:pic2-1}.\label{F:pic-15}}
\end{center}
\end{figure}
%%%%%%%%%%%%%%%%%%%%%%%%%%%%
\begin{definition}[$\Theta_i\,\#\,\Theta_j$]\label{D:operator1}
If the skew diagram of $\lambda/\mu$ is edgewise connected, let $\Phi=(\Theta_1,\Theta_2,\ldots$ $\Theta_g)$ be an outside nested decomposition of skew shape $\lambda/\mu$, and let $H(\Phi)$ be the thickened cutting strip with respect to $\Phi$. For each thickened strip $\Theta_i$ in $\Phi$, if $c_i$ is the content of the starting box of $\Theta_i$, the starting box $P(\Theta_i)$ of $\Theta_i$ is given as below:
\begin{enumerate}
\item $p(\Theta_i)=[c_i]$ if the starting box is not a special corner of $\Phi$;
\item $p(\Theta_i)=[c_i,+]$ if the starting box is a special corner of $\Phi$ and an upper corner of $\Theta_i$;
\item $p(\Theta_i)=[c_i,-]$ if the starting box is a special corner of $\Phi$ and a lower corner of $\Theta_i$.
\end{enumerate}
Likewise, we denote the ending box of $\Theta_i$ by $q(\Theta_i)$ if we replace $p(\Theta_i)$ by $q(\Theta_i)$ and replace the starting box by the ending box from the above notations. Then $\Theta_i$ forms a segment of the thickened cutting strip $H(\Phi)$ starting with the box $p(\Theta_i)$ and ending with the box $q(\Theta_i)$, which is denote by $[p(\Theta_i),q(\Theta_i)]$. We may extend the notion to $[p(\Theta_j),q(\Theta_i)]$ in the following way:
\begin{enumerate}
\item if $c_j<c_i$ or $p(\Theta_j)=q(\Theta_i)$, then $[p(\Theta_j),q(\Theta_i)]$ is a segment of $H(\Phi)$ starting with the box $p(\Theta_j)$ and ending with the box $q(\Theta_i)$;
\item if $p(\Theta_j)$ and $q(\Theta_i)$ are in the same diagonal of $H(\Phi)$, or $c_j=c_i+1$, then $[p(\Theta_j),q(\Theta_i)]=\varnothing$;
\item if $c_j>c_i+1$, then $[p(\Theta_j),q(\Theta_i)]$ is undefined.
\end{enumerate}
For any two thickened strips $\Theta_i$ and $\Theta_j$ of $\Phi$, the thickened strip $\Theta_i\#\Theta_j$ is defined as $[p(\Theta_j),q(\Theta_i)]$.
\end{definition}
\begin{remark}
We only need to deal with the outside nested thickened strip decompositions of an edgewise connected skew diagram because the Schur function of any edgewise disconnected diagram is a product of Schur functions of edgewise connected components.
\end{remark}
\begin{remark}
Since $\Phi$ is an outside nested decomposition, we can identify every thickened strip $\Theta_i$ as a segment of $H(\Phi)$ starting with the box $p(\Theta_i)$ and ending with the box $q(\Theta_i)$.
\end{remark}
\begin{example}\label{E:op1}
Consider the outside nested decomposition $\Phi=(\Theta_1,\Theta_2,\Theta_3)$ in Figure~\ref{F:pic2-01}, one has $\Theta_1\#\Theta_2=[p(\Theta_2),q(\Theta_1)]=[[-3,+],[5]]=(5,5,5,4,4)/(4,3,3,2)$, that is, a segment of the thickened cutting strip $H(\Phi)$ in Figure~\ref{F:pic-15} starting with box $[-3,+]$ and ending with box with content $[5]$. Similarly, the thickened strips obtained by the operation $\#$ are given below:
\[
\begin{array}{lll}
\Theta_1\#\Theta_3=(4,4,4,3,3,1)/(3,2,2,1), & \Theta_2\#\Theta_1=(2,2), &\\
\Theta_2\#\Theta_3=(4,4,3,3,1)/(2,2,1), & \Theta_3\#\Theta_1=\varnothing, & \Theta_3\#\Theta_2=(4,4)/(2).
\end{array}
\]
\end{example}
%%%%%%%%%%%%%%%%%
If $\Phi$ is an outside decomposition where all thickened strips $\Theta_i$ are strips, then the starting box of any strip $\Theta_i$ and the ending box of any strip $\Theta_j$ are $p(\Theta_i)=[c_i]$ and $q(\Theta_j)=[c_j]$.
\begin{remark}
In \cite{HG} Hamel and Goulden noticed that any order of the strips in an outside decomposition play the same role. Chen, Yan and Yang \cite{CYY} also showed Hamel and Goulden's theorem in terms of the canonical order of strips and our extension (Theorem~\ref{T:thick}) also works for any total order of the thickened strips in an outside nested decomposition.
\end{remark}
%%%%%%%%%%%%%%%%%%%%%%%%%%%
\section{Proof of Theorem~\ref{T:thick} and Corollary~\ref{C:thick2}}\label{S:proof1}
Since it is convenient to construct an involution in the context of lattice paths, we choose to represent semistandard Young tableaux of thickened strip shape in the language of lattice paths. Our proof of Theorem~\ref{T:thick} consists of three main steps.

In the first step we build a one-to-one correspondence between semistandard Young tableaux of thickened strip shape to {\em double lattice paths}, which is based on a bijection between semistandard Young tableaux of strip shape and lattice paths in \cite{HG}. In the second step we introduce the {\em separable} $g$-tuples of double lattice paths and show that the generating function of all separable $g$-tuples of double lattice paths is $p_{1^r}(X)s_{\lambda/\mu}(X)$. In the last step we will construct a sign-reversing and weight-preserving involution $\omega$ on all non-separable $g$-tuples of double lattice paths, so that only the separable ones contribute to the determinant $\det[s_{\Theta_i\#\Theta_j}(X)]_{i,j=1}^g$ in Theorem~\ref{T:thick}.

We will prove Corollary~\ref{C:thick2} by using the exponential specializations of the Schur functions and power sum symmetric functions.
%%%%%%%%%%%%%%%%%%%%%%%%%%%%%%%%%%%%%%
\subsection{From Semistandard Young tableaux to double lattice paths}
First we recall that $H(\Phi)$ is the thickened cutting strip which corresponds to $\Phi$ (see Definition~\ref{D:tcs}) and $\Theta_i\#\Theta_j$ is given in Definition~\ref{D:operator1}. For any $i,j$, we will introduce the double lattice path $P(u_j,v_i)$ in Definition~\ref{D:plattice}.
%%%%%%%%%%%%%%%%%%%%%%%%%%%%%%%%%%%%%%%%%%%%
\begin{definition}[double lattice paths]\label{D:plattice}
Under the assumption of Theorem~\ref{T:thick}, for every $i,j$, the {\em double lattice paths} with respect to $\Phi$, are pairs $P(u_j,v_i)=(p_{ji}^+,p_{ji}^-)$ of lattice paths where $p_{ji}^+$ and $p_{ji}^-$ start at $u_j$ and end at $v_i$. The starting point $u_j$ and the ending point $v_i$ are fixed as below:
\begin{enumerate}
\item if the starting box $(s,t)$ of $\Theta_j$ is a common special corner of $\Phi$, and
    \begin{itemize}
    \item[] if box $(s,t)$ is a lower corner of $\Theta_j$, then $u_j=(t-s,1)$;
    \item[] otherwise if box $(s,t)$ is an upper corner of $\Theta_j$, then $u_j=(t-s,\infty)$;
    \end{itemize}
\item[] if the starting box $(s,t)$ of $\Theta_j$ is not a common special corner of $\Phi$, and
    \begin{itemize}
    \item[] if box $(s,t)$ is on the left perimeter of the skew shape $\lambda/\mu$, then $u_j=(t-s,1)$;
    \item[] otherwise if box $(s,t)$ is only on the bottom perimeter of the skew shape $\lambda/\mu$, then $u_j=(t-s,\infty)$;
    \end{itemize}
\item if the ending box $(\mu,\nu)$ of $\Theta_i$ is a common special corner of $\Phi$, and
    \begin{itemize}
    \item[] if box $(\mu,\nu)$ is a lower corner of $\Theta_i$, then $v_i=(\nu-\mu+1,1)$;
    \item[] otherwise if box $(\mu,\nu)$ is an upper corner of $\Theta_i$, then $v_i=(\nu-\mu+1,\infty)$;
    \end{itemize}
\item[] if the ending box $(\mu,\nu)$ of $\Theta_i$ is not a common special corner of $\Phi$, and
    \begin{itemize}
    \item[] if box $(\mu,\nu)$ is on the right perimeter of the skew shape $\lambda/\mu$, then $v_i=(\nu-\mu+1,\infty)$;
    \item[] otherwise if box $(\mu,\nu)$ is only on the top perimeter of the skew shape $\lambda/\mu$, then $v_i=(\nu-\mu+1,1)$.
    \end{itemize}
\end{enumerate}
Furthermore, the lattice paths $p_{ji}^+$ and $p_{ji}^-$ consist of four types of steps: an up-vertical step $\uparrow(0,1)$, a down-vertical step $\downarrow(0,-1)$, a horizontal step $\rightarrow(1,0)$ and a diagonal step $\searrow(1,-1)$, which satisfy the conditions
%%%%%%%%%%%%%%%%%%%%%%%%%%%%%%%%%%%%%%%
\begin{enumerate}
\setcounter{enumi}{2}
\item a down-vertical step $(0,-1)$ must not precede an up-vertical step $(0,1)$ and must not precede a horizontal step $(1,0)$;
\item an up-vertical step $(0,1)$ must not precede a down-vertical step $(0,-1)$ and must not precede a diagonal step $(1,-1)$.
\end{enumerate}
Moreover, there is a horizontal step $(1,0)$ of the lattice path $p_{ji}^+$ or $p_{ji}^-$ between lines $x=c$ and $x=c+1$ if one of $(5),(6)$ holds:
\begin{enumerate}
\setcounter{enumi}{4}
    \item a box of content $c-1$ is to the left of a box of content $c$ in $\Theta_i\#\Theta_j$;
    \item the starting box of $\Theta_j$ has content $c$ and $u_j=(c,1)$.
\end{enumerate}
There is a diagonal step $(1,-1)$ of the lattice path $p_{ji}^+$ or $p_{ji}^-$ between lines $x=c$ and $x=c+1$ if one of $(7),(8)$ holds:
\begin{enumerate}
\setcounter{enumi}{6}
\item a box of content $c-1$ is right below a box of content $c$ in $\Theta_i\#\Theta_j$;
\item the starting box of $\Theta_j$ has content $c$ and $u_j=(c,\infty)$.
\end{enumerate}
We connect all these non-vertical steps by up-vertical and down-vertical steps so that every non-vertical step of $p_{ji}^+$ is either above or the same as the one of $p_{ji}^-$ between any lines $x=c$ and $x=c+1$.
\end{definition}
%%%%%%%%%%%%%%%%%%%%%%%%%%%%%%%%%%%%%%%%%%%%%%%%%%%%%%%
\begin{remark}
When $\Phi$ is an outside decomposition, for all $i$ and $j$, the double lattice path $(p_{ji}^+,p_{ji}^-)=P(u_j,v_i)$ with respect to $\Phi$ is a lattice path, that is, $p_{ji}^+=p_{ji}^-$ where all steps between any lines $x=c$ and $x=c+1$ are all horizontal or all diagonal; see \cite{HG}.
\end{remark}
Because $\Phi$ is an outside nested decomposition, by $(1)$-$(2)$ in Definition~\ref{D:plattice}, all starting points and all ending points are all different. Once the starting point $u_j$ and the ending point $v_i$ are chosen, the shape of any double lattice path $P(u_j,v_i)$ is fixed, that is, whether any non-vertical step of $P(u_j,v_i)$ is horizontal or diagonal, is determined by $\Theta_i\#\Theta_j$. This allows us to identify the Schur function $s_{\Theta_i\#\Theta_j}(X)$ as the generating function of all weighted double lattice paths from $u_j$ to $v_i$ in Subsection~\ref{ss:invo}.
%From Definition~\ref{D:operator1} we see that $\Theta_i\#\Theta_j$ is determined by its starting box and ending box in $H(\Phi)$. So are the double lattice paths where all non-vertical steps of any double lattice path $P(u_j,v_i)$ with respect to $\Phi$ are determined once the starting point $u_j$ and the ending point $v_i$ are fixed.

For every $i$ and $j$, let $\CMcal{P}(u_j,v_i)$ represent the set of all double lattice paths from $u_j$ to $v_i$, and let $\CMcal{T}_{\Theta_i\#\Theta_j}$ represent the set of all semistandard Young tableaux of thickened strip shape $\Theta_i\#\Theta_j$. We next establish that
%%%%%%%%%%%%%%%%%%%%%%%%%%%%%%%%%%%%%%
\begin{lemma}\label{L:bij3}
There is a bijection $f$ between the set $\CMcal{T}_{\Theta_i\#\Theta_j}$ and the set $\CMcal{P}(u_j,v_i)$.
\end{lemma}
\begin{proof}
If $\Theta_i\#\Theta_j=\varnothing$, according to $(5)$-$(8)$ of Definition~\ref{D:plattice}, $\CMcal{P}(u_j,v_i)$ contains only one double lattice path $P(u_j,v_i)$ that has no non-vertical steps from $u_j$ to $v_i$, which corresponds to the empty tableau from $\CMcal{T}_{\varnothing}$.

%By Definition~\ref{D:operator1} if $c_j=c_i+1$, we have $\Theta_i\#\Theta_j=\varnothing$ and according to Definition~\ref{D:plattice}, the starting point $u_j$ and the ending point $v_i$ must have the same $x$-coordinate $c_j$. So the only double lattice path from $\CMcal{P}(u_j,v_i)$ contains only vertical steps from $u_j$ to $v_i$, which corresponds to the empty tableau from $\CMcal{T}_{\varnothing}$.

%If $p(\Theta_j)$ and $q(\Theta_i)$ are on the same diagonal of $H(\Phi)$, by Definition~\ref{D:operator1} we have $\Theta_i\#\Theta_j=\varnothing$ and according to Definition~\ref{D:plattice}, the $x$-coordinates of $u_j$ and $v_i$ differ by $1$. So the only double lattice path from $\CMcal{P}(u_j,v_i)$ contains a non-vertical step ending at $(c_j+1,1)$ and vertical steps on lines $x=c_j$ or $x=c_j+1$, which corresponds to the empty tableau from $\CMcal{T}_{\varnothing}$.

If $\Theta_i\#\Theta_j$ is undefined, the starting point $u_j$ is on the right hand side of the ending point $v_i$, so by Definition~\ref{D:plattice} there exist no double lattice paths from $u_j$ to $v_i$, that is, the set $\CMcal{P}(u_j,v_i)$ is undefined, which corresponds to the undefined set $\CMcal{T}_{\Theta_i\#\Theta_j}$.

Otherwise, given a semistandard Young tableau $T_{\Theta_i\#\Theta_j}$ of thickened strip shape $\Theta_i\#\Theta_j$, we build the corresponding double lattice path $f(T_{\Theta_i\#\Theta_j})=(p_{ji}^+,p_{ji}^-)=P(u_j,v_i)$ starting with $u_j$ and ending at $v_i$. For every box $\alpha$ in $T_{\Theta_i\#\Theta_j}$, suppose that the box $\alpha$ of content $c$ has entry $q$ in $T_{\Theta_i\#\Theta_j}$. Then we put a horizontal step from $(c,q)$ to $(c+1,q)$ if one of $(1),(2)$ is true:
\begin{enumerate}
\item a box of content $c-1$ is to the left of $\alpha$ in $\Theta_i\#\Theta_j$;
\item $\alpha$ is the starting box of $\Theta_j$ and $u_j=(c,1)$.
\end{enumerate}
We put a diagonal step from $(c,q+1)$ to $(c+1,q)$ if one of $(3),(4)$ is true:
\begin{enumerate}
\setcounter{enumi}{2}
\item a box of content $c-1$ is right below $\alpha$ in $\Theta_i\#\Theta_j$;
\item $\alpha$ is the starting box of $\Theta_j$ and $u_j=(c,\infty)$.
\end{enumerate}
We connect all these non-vertical steps by up-vertical and down-vertical steps. In this way we get a pair $(p_{ji}^+,p_{ji}^-)=P(u_j,v_i)$ of lattice paths where every non-vertical step of $p_{ji}^+$ is either above or the same as the one of $p_{ji}^-$. By construction in the lattice path $p_{ji}^+$ or $p_{ji}^-$, there is no down-vertical step preceding an up-vertical step and there is no up-vertical step preceding a down-vertical step. Since $T_{\Theta_i\#\Theta_j}$ is a semistandard Young tableau, there is no down-vertical step preceding a horizontal step because otherwise, the entries along each row of $T_{\Theta_i\#\Theta_j}$ is not weakly increasing from left to right. Similarly there is no up-vertical step preceding a diagonal step because otherwise, the entries along each column of $T_{\Theta_i\#\Theta_j}$ is not strictly decreasing from bottom to top. So by Definition~\ref{D:plattice} the path $(p_{ji}^+,p_{ji}^-)=P(u_j,v_i)$ is a double lattice path. The map $$f:\CMcal{T}_{\Theta_i\#\Theta_j}\rightarrow\CMcal{P}(u_j,v_i).$$
is a bijection because the above process is reversible.
\end{proof}
In Lemma~\ref{L:bij3} we observe that the point $(c+1,q)$ is the ending point of some non-vertical step of $P(u_j,v_i)$ if and only if a box of content $c$ has entry $q$ in $T_{\Theta_i\#\Theta_j}$ where $f(T_{\Theta_i\#\Theta_j})=P(u_j,v_i)$.
%%%%%%%%%%%%%%%%%%%%%%%%%%%%%%%%%%%%%%%%%%%%%
\begin{example}
For $i=1,2,3$, consider the thickened strip $\Theta_i$ of the skew shape $(6,6,6,4)/(3,1)$ in Figure~\ref{F:pic2-01}, the corresponding double lattice path $P(u_2,v_3)\in \CMcal{P}(u_2,v_3)$ of the thickened strip tableau $T_{\Theta_3\#\Theta_2}$ is given in Figure~\ref{F:pic2-7} where all integers represent the $y$-th coordinates of all ending points from the non-vertical steps in $P(u_2,v_3)$.
We have discussed the shape of $\Theta_3\#\Theta_2$ in Example~\ref{E:op1}. Since the starting box $p(\Theta_2)=[-3,+]$ of $\Theta_3\#\Theta_2$ is an upper corner of $\Theta_2$, according to $(1)$ in Definition~\ref{D:plattice}, the starting point $u_2$ is $(-3,\infty)$ and we put a diagonal step from $(-3,3)$ to $(-2,2)$ in Figure~\ref{F:pic2-6} because of $(8)$ in Definition~\ref{D:plattice}. Similarly, since the ending box of $\Theta_3\#\Theta_2$ is $q(\Theta_3)=[1]$, the ending point $v_3$ is $(2,1)$.

In addition, the corresponding double lattice path $P(u_1,v_3)\in\CMcal{P}(u_1,v_3)$ of the empty thickened strip tableau $T_{\Theta_3\#\Theta_1}=T_{\varnothing}$ consists of only vertical steps from $u_1=(2,\infty)$ to $v_3=(2,1)$.
\end{example}
%%%%%%%%%%%%%%%%%%%%%%%%%%%%%%%%%%%%%%%%%%%%%%%%%
\begin{figure}[htbp]
\begin{center}
\includegraphics[scale=0.85]{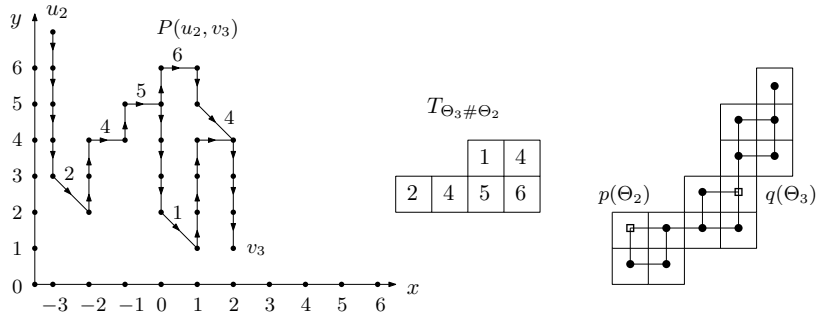}
\caption{A double lattice path $P(u_2,v_3)$ (left), the corresponding thickened strip tableau $T_{\Theta_3\#\Theta_2}=f^{-1}(P(u_2,v_3))$ (middle) and the thickened cutting strip $H(\Phi)$ where the starting box $p(\Theta_2)=[-3,+]$ and the ending box $q(\Theta_3)=[1]$ of $\Theta_3\#\Theta_2$ are marked with empty squares (right). \label{F:pic2-7}}
\end{center}
\end{figure}
%%%%%%%%%%%%%%%%%%%%%%%%%%%%%%%%%%%%%%%%%%%%%

With the help of Lemma~\ref{L:bij3} we can establish the relation between semistandard Young tableaux of skew shape $\lambda/\mu$ and $g$-tuples of double lattice paths. Under the assumption of Theorem~\ref{T:thick}, for any $\pi\in S_g$, we write $\pi=\pi_1\pi_2\cdots \pi_g$ and consider any $g$-tuple (\ref{E:good}) of double lattice paths, where the starting points and the ending points of all steps are called {\em points} of (\ref{E:good}).
%%%%%%%%%%%%%%%%%%%%%%%%%%%%%%%%%%
\begin{definition}[non-crossing]\label{D:noncr}
Consider a $g$-tuple
\begin{align}\label{E:good}
(P(u_{\pi_1},v_{1}),P(u_{\pi_2},v_{2}),\ldots, P(u_{\pi_g},v_{g}))
\end{align}
of double lattice paths where $P(u_{\pi_i},v_i)\in \CMcal{P}(u_{\pi_i},v_i)$. Then (\ref{E:good}) is {\em non-crossing} if for any $i$ and $j$, $P(u_{\pi_i},v_i)$ and $P(u_{\pi_j},v_j)$ are non-crossing. This holds if and only if
\begin{enumerate}
\item $P(u_{\pi_i},v_i)$ and $P(u_{\pi_j},v_j)$ are non-intersecting, that is, have no points in common;
\item $P(u_{\pi_j},v_j)$ is on the top side of $P(u_{\pi_i},v_i)$ and they have some points in common, where each common point $(c+1,q)$ occurs only when one diagonal step of $P(u_{\pi_j},v_j)$ and one horizontal step of $P(u_{\pi_i},v_i)$ end at the same point $(c+1,q)$. %and neither $P(u_{\pi_i},v_i)$ nor $P(u_{\pi_j},v_j)$ has only one non-vertical step.
\end{enumerate}
Otherwise $P(u_{\pi_i},v_i)$ and $P(u_{\pi_j},v_j)$ are crossing and (\ref{E:good}) is crossing. If (\ref{E:good}) is non-crossing, we call every common point of any two double lattice paths in (\ref{E:good}) a {\em touchpoint} of (\ref{E:good}).
\end{definition}
%%%%%%%%%%%%%%%%%%%%%%%%%%%%%%%%%%%%%%%%%%
\begin{remark}
When $\pi=\mbox{id}=(1)(2)\cdots(g)$, two double lattice paths $P(u_i,v_i)$ and $P(u_j,v_j)$ are non-crossing if and only if two semistandard Young tableaux $T_{\Theta_i}$ and $T_{\Theta_j}$ are disjoint or have the same entry in every common special corner of $\Theta_i$ and $\Theta_j$, where $f(T_{\Theta_q})=P(u_q,v_q)$ for $q\in\{i,j\}$.
\end{remark}
%%%%%%%%%%%%%%%%%%%%%%%%%%%%%%%%%%%%%%%%%%
\begin{example}
The triple $(P(u_1,v_1),P(u_2,v_2),P(u_3,v_3))$ of double lattice paths given in Figure~\ref{F:pic2-6} where the $y$-coordinates of $u_1,v_1,u_2$ are all infinity, is non-crossing and all touchpoints have coordinates $(-2,3),(1,4),(4,3)$.
\end{example}
%%%%%%%%%%%%%%%%%%%%%%%%%%%%%%%%%%%%%%%%%%
\begin{lemma}\label{L:compatible}
If a $g$-tuple (\ref{E:good}) of double lattice path is non-crossing, then $\pi$ must be the identity permutation, that is, $\pi=\mbox{id}=(1)(2)\cdots(g)$.
\end{lemma}
\begin{proof}
We shall prove the equivalent statement, namely, if $\pi\in S_g$ and $\pi\ne\mbox{id}$, then any $g$-tuple (\ref{E:good}) of double lattice path is crossing.

First we consider a total order $\prec$ of all starting points $u_1,u_2,\ldots,u_g$ and a total order $\prec$ of all ending points $v_1,v_2,\ldots,v_g$ of the double lattice paths. For every $i$, let $x(u_i)$ and $y(u_i)$ denote the $x$-th coordinate and $y$-th coordinate of point $u_i$, similarly for $x(v_i)$ and $y(v_i)$. We recall that $y(u_i),y(v_i)\in \{1,\infty\}$ according to Definition~\ref{D:plattice}. We define $u_s\prec u_i$ if and only if one of the following conditions is true:
\begin{enumerate}
\item $\infty=y(u_s)>y(u_i)=1$;
\item $y(u_s)=y(u_i)=\infty$ and $x(u_s)>x(u_i)$;
\item $y(u_s)=y(u_i)=1$ and $x(u_s)<x(u_i)$.
\end{enumerate}
We define $v_s\prec v_i$ if and only if one of the following conditions is true:
\begin{enumerate}
\setcounter{enumi}{3}
\item $\infty=y(v_s)>y(v_i)=1$;
\item $y(v_s)=y(v_i)=\infty$ and $x(v_s)<x(v_i)$;
\item $y(v_s)=y(v_i)=1$ and $x(v_s)>x(v_i)$.
\end{enumerate}
We claim that for any $i$ and $s$, $u_s\prec u_i$ holds if and only if $v_s\prec v_i$ holds. The essential reason for this is the fact that $\Phi$ is an outside thickened strip decomposition (Definition~\ref{D:todec}), so when we read the boxes on the bottom perimeter and the left perimeter of the skew shape $\lambda/\mu$ in the right-to-left and bottom-to-top order, the starting box of $\Theta_s$ comes earlier than the starting box of $\Theta_i$ if and only if $u_s\prec u_i$ holds. Since one thickened strip is on the right side or the bottom side of the other thickened strip; see Definition~\ref{D:todec}, when we read the boxes on the right perimeter and the top perimeter of the skew shape $\lambda/\mu$ in the bottom-to-top and right-to left order, the ending box of $\Theta_s$ comes earlier than the ending box of $\Theta_i$ if and only if $v_s\prec v_i$ holds. This implies that for any $i$ and $s$, $u_s\prec u_i$ holds if and only if $v_s\prec v_i$ holds.

Second, for any $\pi$ such that $\mbox{id}\ne\pi\in S_g$, there exist two integers $s$ and $t$ such that $u_{\pi_s}\prec u_{\pi_t}$ and $v_t\prec v_s$ because otherwise it contradicts the assumption $\pi\ne \mbox{id}$. We wish to show that $P(u_{\pi_s},v_s)$ and $P(u_{\pi_t},v_t)$ are crossing, which can be proved by discussing all cases when one of the previous conditions $(1)$-$(3)$ for $u_{\pi_s}\prec u_{\pi_t}$ is true, and one of the previous conditions $(4)$-$(6)$ for $u_{t}\prec u_{s}$ is true. So we conclude that if $\pi\in S_g$ and $\pi\ne \mbox{id}$, then (\ref{E:good}) is crossing.
\end{proof}
%%%%%%%%%%%%%%%%%%%%%%%%%%%%%%%%%%%%%%%%%%
\begin{remark}
Here note that we need the condition that $\Phi$ is an outside thickened strip decomposition. Lemma~\ref{L:compatible} actually verifies the condition of Stembridge's theorem on the non-intersecting lattice paths \cite{Stembridge}. Though Stembridge considered only the non-intersecting lattice paths, his theorem is still applicable to the non-crossing double lattice paths.
\end{remark}
%%%%%%%%%%%%%%%%%%%%%%%%%%%%%%%%%%%%%%%%%%
\begin{proposition}\label{L:bij5}
Under the assumption of Theorem~\ref{T:thick}, there is a bijection between semistandard Young tableaux of skew shape $\lambda/\mu$ and non-crossing $g$-tuples of double lattice paths with $r$ touchpoints.
\end{proposition}
%%%%%%%%%%%%%%%%%%%%%%%%%%%%%%%%%%%%%%%%%%
\begin{proof}
In view of Lemma~\ref{L:compatible}, we shall establish a bijection between semistandard Young tableaux of skew shape $\lambda/\mu$ and non-crossing $g$-tuples
\begin{align}\label{E:noncr}
(P(u_1,v_1),P(u_2,v_2),\ldots,P(u_g,v_g))
\end{align}
of double lattice paths with $r$ touchpoints where $P(u_i,v_i)\in\CMcal{P}(u_i,v_i)$ for every $i$.

For a semistandard Young tableau $T$ of the skew shape $\lambda/\mu$, we can express $T$ as a $g$-tuple $(T_{\Theta_1},T_{\Theta_2},\ldots,T_{\Theta_g})$ of thickened strip tableaux where $T_{\Theta_i}$ is $T$ that is restricted to the thickened strip shape $\Theta_i$. Combining the bijection $f$ in Lemma~\ref{L:bij3}, one gets the $g$-tuple (\ref{E:noncr}) of double lattice paths where $P(u_i,v_i)=f(T_{\Theta_i})$ is a double lattice path from $u_i$ to $v_i$ and the fact that (\ref{E:noncr}) is non-crossing follows from the fact that all entries of boxes on the same diagonal of $T$ are strictly increasing from the top-left side to the bottom-right side. The map $(T_{\Theta_1},T_{\Theta_2},\ldots,T_{\Theta_g})\mapsto (\ref{E:noncr})$ is a bijection because, for any $i$ and $j$, two double lattice paths $P(u_{i},v_i)$ and $P(u_{j},v_{j})$ are non-intersecting if and only if two thickened strip tableaux $T_{\Theta_i}$ and $T_{\Theta_j}$ are disjoint. Furthermore, $P(u_{j},v_j)$ is on the top side of $P(u_{i},v_i)$ such that the diagonal step of $P(u_{j},v_j)$ and the horizontal step of $P(u_{i},v_i)$ end at the same point $(c+1,q)$ if and only if the box of content $c$ and with entry $q$ in $T$, is an upper corner of $\Theta_j$ and a lower corner of $\Theta_i$. Since there are $r$ common special corners of $\Phi$, there are $r$ touchpoints of (\ref{E:noncr}).
\end{proof}
%%%%%%%%%%%%%%%%%%%%%%%%%%%%%%%%%%%%%%%%%%%%%%%%
\begin{example}
Consider the semistandard Young tableau $T=(T_{\Theta_1},T_{\Theta_2},T_{\Theta_3})$ of skew shape $(6,6,6,4)/(3,1)$ in Figure~\ref{F:pic2-9}, the corresponding triple of double lattice paths $P(u_i,v_i)=f(T_{\Theta_i})$ is displayed in Figure~\ref{F:pic2-6} where the $y$-coordinates of $u_1,v_1,u_2$ are all infinity.

In addition, from the non-crossing triple $(P(u_1,v_1),P(u_2,v_2),P(u_3,v_3))$ of double lattice paths in Figure~\ref{F:pic2-6}, one has $u_1\prec u_2 \prec u_3$ and $v_1\prec v_2\prec v_3$. So for instance, any double lattice paths $P(u_1,v_3)$ and $P(u_2,v_2)$ are crossing because $u_1\prec u_2$ but $v_2\prec v_3$.
\end{example}
%%%%%%%%%%%%%%%%%%%%%%%%%%%%%%%%%%%%%%
\begin{figure}[htbp]
\begin{center}
\includegraphics[scale=0.8]{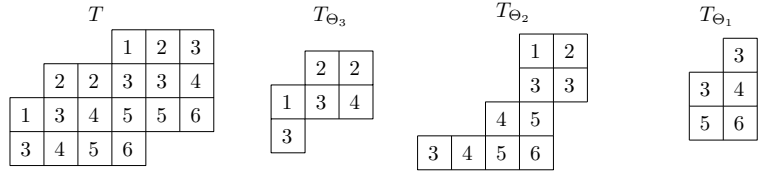}
\caption{A semistandard Young tableau $T$ which is equivalent to a triple $(T_{\Theta_1},T_{\Theta_2},T_{\Theta_3})$ of thickened strip tableaux and $\Phi=(\Theta_1,\Theta_2,\Theta_3)$ is given in Figure~\ref{F:pic2-01}.
\label{F:pic2-9}}
\end{center}
\end{figure}
%%%%%%%%%%%%%%%%%%%%%%%%%%%%%%%%%%%%%%
\begin{figure}[htbp]
\begin{center}
\includegraphics[scale=1.0]{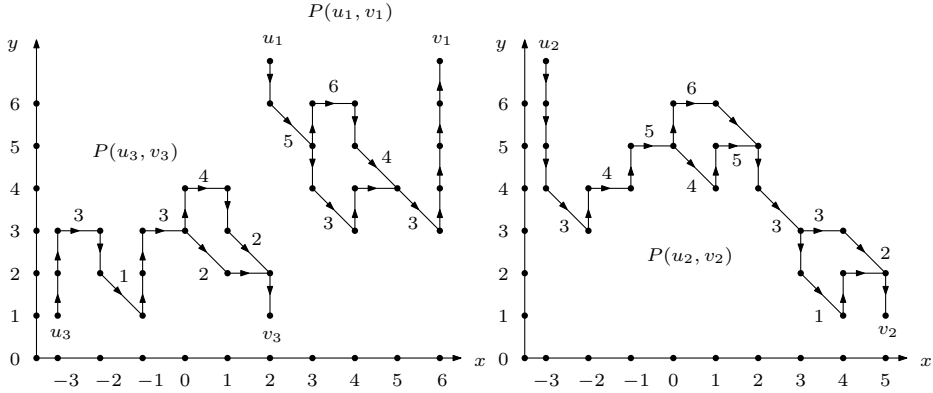}
\caption{Three double lattice paths where each $P(u_i,v_i)=f(T_{\Theta_i})$ uniquely corresponds to the semistandard Young tableau $T_{\Theta_i}$ in Figure~\ref{F:pic2-9}. \label{F:pic2-6}}
\end{center}
\end{figure}
%%%%%%%%%%%%%%%%%%%%%%%%%%%%%%%%%%%%%%
\subsection{Count the separable sequences of double lattice paths}
For a $g$-tuple
\begin{align}\label{E:good3}
\tilde{P}=(\tilde{P}(u_{\pi_1},v_1),\tilde{P}(u_{\pi_2},v_2),\ldots,
\tilde{P}(u_{\pi_g},v_g))
\end{align}
of double lattice paths where $\tilde{P}(u_{\pi_i},v_i)\in\CMcal{P}(u_{\pi_i},v_i)$ for every $i$, we will describe a separable $g$-tuple of double lattice paths and our main task is to establish the bijection in Proposition~\ref{P:nonspr}, from which it follows that the generating function of all weighted separable $g$-tuples of double lattice paths is $p_{1^r}(X)s_{\lambda/\mu}(X)$; see Subsection~\ref{ss:invo}.
%%%%%%%%%%%%%%%%%%%%%%%%%%%%%%%%%%%%%%%%%%%%%%%%%%%%%%%%%%%%
\begin{definition}[non-separable at a single point]\label{D:spe1}
For all $c$ such that neither $c$ nor $c-1$ is the content of some special corner of $\Phi$, we say that two double lattice paths $\tilde{P}(u_{\pi_i},v_i)$ and $\tilde{P}(u_{\pi_j},v_j)$ are {\em non-separable at the point $(c,y)$} if and only if $\tilde{P}(u_{\pi_i},v_i)$ intersects $\tilde{P}(u_{\pi_j},v_j)$ at the point $(c,y)$, that is, $\tilde{P}(u_{\pi_i},v_i)$ and $\tilde{P}(u_{\pi_j},v_j)$ have the point $(c,y)$ in common.

For a $g$-tuple $\tilde{P}$ (see (\ref{E:good3})) of double lattice paths, we say that $\tilde{P}$ is non-separable at a single point if there exist two double lattice paths in $\tilde{P}$ such that they are non-separable at a single point. Otherwise we say that $\tilde{P}$ is separable at all single points.
\end{definition}
\begin{remark}
The point $(c,y)$ in Definition~\ref{D:spe1} is not a touchpoint because if the point $(c,y)$ is a touchpoint, then $c-1$ must be the content of some special corner of $\Phi$, which is impossible according to the assumption in Definition~\ref{D:spe1}. When the outside nested decomposition $\Phi$ is an outside decomposition, there is no special corners of $\Phi$ and any double lattice path is a lattice path. So in this case any two double lattice paths are non-separable at the point $(c,y)$ if and only if two lattice paths are intersecting at the point $(c,y)$.
\end{remark}
\begin{definition}[$c$-point, $\CMcal{C}$-pair]\label{D:spe2}
For all $c$ such that $c$ is the content of some special corner of $\Phi$, and for all $i$, if $\tilde{P}(u_{\pi_i},v_i)$ has a point on line $x=c$, we consider the unique {\em $c$-point} of $\tilde{P}(u_{\pi_i},v_i)$, which is the ending point of the non-vertical step of $\tilde{P}(u_{\pi_i},v_i)$ between lines $x=c-1$ and $x=c$, or the starting point of $\tilde{P}(u_{\pi_i},v_i)$ on line $x=c$. If $i\ne j$, the $c$-point $(c,y_2)$ of $\tilde{P}(u_{\pi_j},v_j)$ is above the one $(c,y_1)$ of $\tilde{P}(u_{\pi_i},v_i)$, that is, $y_1<y_2$, and there is no other $c$-points between $(c,y_1)$ and $(c,y_2)$.
Then we call $([y_1,y_2],c)$ a {\em $\CMcal{C}$-pair}.
\end{definition}
\begin{remark}
By Definition~\ref{D:spe2} it is clear that the number of $\CMcal{C}$-pairs and the number of touchpoints of a non-crossing $g$-tuple of double lattice paths are the same, which are both equal to the number of common special corners of $\Phi$.
\end{remark}
\begin{example}
For a triple $(\tilde{P}(u_1,v_1),\tilde{P}(u_2,v_2),\tilde{P}(u_3,v_3))$ of double lattice paths in Figure~\ref{F:pic3-4}, and for $c=-3$, the $(-3)$-points of $\tilde{P}(u_2,v_2)$ and $\tilde{P}(u_3,v_3)$ are $(-3,\infty)$ and $(-3,1)$. So $([1,\infty],-3)$ is a $\CMcal{C}$-pair. Similarly, the $(0)$-points of $\tilde{P}(u_2,v_2)$ and $\tilde{P}(u_3,v_3)$ are $(0,5)$ and $(0,3)$, as well as the $(3)$-points of $\tilde{P}(u_1,v_1)$ and $\tilde{P}(u_2,v_2)$ are $(3,5)$ and $(3,3)$. Consequently the triple of double lattice paths $\tilde{P}(u_i,v_i)$ contains three $\CMcal{C}$-pairs, which are $([1,\infty],-3)$, $([3,5],0)$ and $([3,5],3)$.
\end{example}
%%%%%%%%%%%%%%%%%%%%%%%%%%%%%%%%%%%%%%%%%%%%%%%%
\begin{figure}[htbp]
\begin{center}
\includegraphics[scale=1.0]{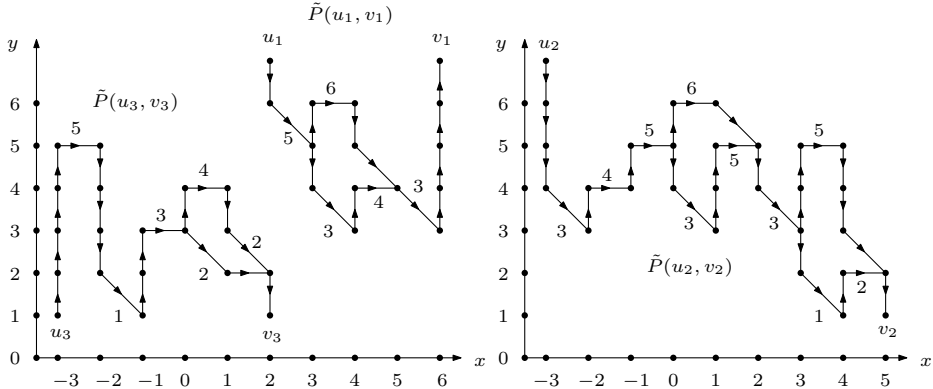}
\caption{A separable triple of double lattice paths $\tilde{P}(u_i,v_i)$.
\label{F:pic3-4}}
\end{center}
\end{figure}
%%%%%%%%%%%%%%%%%%%%%%%%%%%%%%%%%%%%%%%%%%%%%%%%
We define that $\tilde{P}_c(u_j,v_i)[a,b]$ is constructed from $\tilde{P}(u_j,v_i)\in\CMcal{P}(u_j,v_i)$ by the following steps:
\begin{itemize}
\item remove the vertical steps on lines $x=c$ and $x=c+1$;
\item shift the non-vertical step ending at $(c+1,a)$ to the non-vertical step ending at $(c+1,b)$;
\item add vertical steps on lines $x=c$ and $x=c+1$ to connect with the new non-vertical step.
\end{itemize}
It should be noted that $\tilde{P}_c(u_j,v_i)[a,b]$ may not be a double lattice path.
%%%%%%%%%%%%%%%%%%%%%%%%%%%%%%%%%%%%%
\begin{definition}[non-separable at a $\CMcal{C}$-pair]\label{D:rep}
For all $\CMcal{C}$-pairs $([y_1,y_2],c)$ of a $g$-tuple $\tilde{P}$ (see (\ref{E:good3})) of double lattice paths, suppose that $(c,y_2),(c,y_1)$ are the $c$-points of $\tilde{P}(u_{\pi_j},v_j)$ and $\tilde{P}(u_{\pi_i},v_i)$, the diagonal step of $\tilde{P}(u_{\pi_j},v_j)$ ends at $(c+1,b)$ and the horizontal step of $\tilde{P}(u_{\pi_i},v_i)$ ends at $(c+1,a)$. Then we say that $\tilde{P}(u_{\pi_i},v_i)$ and $\tilde{P}(u_{\pi_j},v_j)$ are {\em non-separable at a $\CMcal{C}$-pair $([y_1,y_2],c)$} if and only if neither $\tilde{P}_c(u_{\pi_j},v_j)[b,a]$ nor $\tilde{P}_c(u_{\pi_i},v_i)[a,b]$ is a double lattice path.

For a $g$-tuple $\tilde{P}$ (see (\ref{E:good3})) of double lattice paths, we say that $\tilde{P}$ is non-separable at a $\CMcal{C}$-pair if there exist two double lattice paths in $\tilde{P}$ such that they are non-separable at a $\CMcal{C}$-pair. Otherwise we say that $\tilde{P}$ is separable at all $\CMcal{C}$-pairs.
\end{definition}
%%%%%%%%%%%%%%%%%%%%%%%%%%%%%%%%%%%%%
\begin{definition}[separable double lattice paths]\label{D:sdlp}
For a $g$-tuple $\tilde{P}$ (see (\ref{E:good3})) of double lattice paths, we say that $\tilde{P}$ is {\em separable} if and only if $\tilde{P}$ is neither non-separable at any single point nor non-separable at any $\CMcal{C}$-pair.
\end{definition}
%%%%%%%%%%%%%%%%%%%%%%%%%%%%%%%%%%%%%
\begin{example}
The triple $\tilde{P}=(\tilde{P}(u_1,v_1),\tilde{P}(u_2,v_2),\tilde{P}(u_3,v_3))$ of double lattice paths in Figure~\ref{F:pic3-4} is separable. For all $c$ such that $c\in \{-1,2,5,6\}$, any two double lattice paths from $\tilde{P}$ are not intersecting on line $x=c$. For the $\CMcal{C}$-pair $([1,\infty],-3)$, we find that between lines $x=-3$ and $x=-2$, the diagonal step of $\tilde{P}(u_2,v_2)$ ends at $(-2,3)$, the horizontal step of $\tilde{P}(u_3,v_3)$ ends at $(-2,5)$, and $\tilde{P}_{-3}(u_3,v_3)[5,3]$ is a double lattice path. Similarly, $\tilde{P}_{0}(u_2,v_2)[3,4]$ and $\tilde{P}_{3}(u_2,v_2)[5,3]$ are double lattice paths.
\end{example}
%%%%%%%%%%%%%%%%%%%%%%%%%%%%%%%%%%%%%%%%%%%%%%%%%%
\begin{proposition}\label{P:nonspr}
Under the assumption of Theorem~\ref{T:thick}, given any fixed total order of all points in the $2$-dimensional $\mathbb{N}\times\mathbb{N}$ grid, there is a bijection between all separable $g$-tuples of double lattice paths with $r$ distinct $\CMcal{C}$-pairs and all pairs $(\{a_i\}_{i=1}^r,P)$ where $\{a_i\}_{i=1}^r$ is a sequence of $r$ positive integers and $P$ is a non-crossing $g$-tuple of double lattice paths with $r$ distinct touchpoints.
\end{proposition}
%%%%%%%%%%%%%%%%%%%%%%%%%%%%%%%%%%%%%%%%%%%%%%%%%%
\begin{proof}
From Lemma~\ref{L:compatible} we know that if $P$ is a non-crossing $g$-tuple (\ref{E:good}) of double lattice paths, then one gets
\begin{align}\label{E:good2}
P=(P(u_1,v_1),P(u_2,v_2),\ldots,P(u_g,v_g))
\end{align}
where $P(u_i,v_i)\in \CMcal{P}(u_i,v_i)$ for every $i$. First we establish that all pairs $(\{a_i\}_{i=1}^r,P)$ where $\{a_i\}_{i=1}^r$ is a sequence of $r$ positive integers and $P$ is a non-crossing $g$-tuple (\ref{E:good2}) of double lattice paths with $r$ distinct touchpoints, are in bijection with all separable $g$-tuples
\begin{align}\label{E:good4}
\tilde{P}=(\tilde{P}(u_1,v_1),\tilde{P}(u_2,v_2),\ldots,\tilde{P}(u_g,v_g))
\end{align}
of double lattice paths with $r$ distinct $\CMcal{C}$-pairs where $\tilde{P}(u_i,v_i)\in \CMcal{P}(u_i,v_i)$ for every $i$. That is, to prove the map
$(\{a_i\}_{i=1}^r,P)\mapsto \tilde{P}$
is a bijection. Second, we prove that for any separable $g$-tuple (\ref{E:good3}) of double lattice paths, $\pi$ must be the identity permutation.

Given such a pair $(\{a_i\}_{i=1}^r,P)$, by assumption all double lattice paths $P(u_i,v_i)$ have $r$ distinct touchpoints, suppose that $(b_i+1,d_i)$ is the coordinate of the $i$-th touchpoint with respect to any total order of all points in the $2$-dimensional $\mathbb{N}\times\mathbb{N}$ grid. Then for all $i$ such that $1\le i\le r$, assume that the diagonal step of $P(u_{s_i},v_{s_i})$ and the horizontal step of $P(u_{t_i},v_{t_i})$ intersect at the point $(b_i+1,d_i)$, we shall insert $a_i$ to the double lattice paths according to the following steps:
\begin{enumerate}
\item if $P_{b_i}(u_{s_i},v_{s_i})[d_i,a_i]$ is a double lattice path, then we replace the non-vertical steps between lines $x=b_i$ and $x=b_i+1$, together with the vertical steps on lines $x=b_i$ and $x=b_i+1$ of $P(u_{s_i},v_{s_i})$ by the ones from $P_{b_i}(u_{s_i},v_{s_i})[d_i,a_i]$;
\item otherwise, $a_i>d_i$ and we replace the non-vertical steps between lines $x=b_i$ and $x=b_i+1$, together with the vertical steps on lines $x=b_i$ and $x=b_i+1$ of $P(u_{t_i},v_{t_i})$ by the ones from $P_{b_i}(u_{t_i},v_{t_i})[d_i,a_i]$.
\end{enumerate}
We choose $\tilde{P}(u_i,v_i)$ to be the double lattice path $P(u_i,v_i)$ after inserting all integers $a_1,a_2,\ldots,a_r$ to the $g$-tuple $P$ (see (\ref{E:good2})) of double lattice paths. So it suffices to prove the $g$-tuple $\tilde{P}$ (see (\ref{E:good4})) of double lattice paths is separable.

We observe that for all $c$ such that neither $c$ nor $c-1$ is the content of some special corner of $\Phi$, all non-vertical steps of $\tilde{P}$ between lines $x=c$ and $x=c+1$ are the same as the ones of $P$, so any two double lattice paths from $\tilde{P}$ are separable at any single point since $P$ is non-intersecting between lines $x=c$ and $x=c+1$. In addition, we notice that all $\CMcal{C}$-pairs of $\tilde{P}$ and all $\CMcal{C}$-pairs of $P$ are the same. So we claim that for all $i$, $\tilde{P}(u_{s_i},v_{s_i})$ and $\tilde{P}(u_{t_i},v_{t_i})$ are separable at any $\CMcal{C}$-pair $([y_1,y_2],b_i)$. If not, by Definition~\ref{D:rep} it would contradict the facts that the point $(b_i+1,d_i)$ is a touchpoint of $P(u_{s_i},v_{s_i})$ and $P(u_{t_i},v_{t_i})$ and $P$ is non-crossing.

Conversely, for a separable $g$-tuple $\tilde{P}$ of double lattice paths and for all $\CMcal{C}$-pairs $([y_{1},y_{2}],b_i)$ of $\tilde{P}$, if the horizontal step of $\tilde{P}(u_{t_i},v_{t_i})$ ends at point $(b_i+1,b)$ and the diagonal step of $\tilde{P}(u_{s_i},v_{s_i})$ ends at point $(b_i+1,a)$. Since $\tilde{P}$ is separable, according to Definition~\ref{D:sdlp}, one of $\tilde{P}_{b_i}(u_{t_i},v_{t_i})[b,a]$ and $\tilde{P}_{b_i}(u_{s_i},v_{s_i})[a,b]$ must be a double lattice path.
\begin{enumerate}
\setcounter{enumi}{2}
\item If $\tilde{P}_{b_i}(u_{s_i},v_{s_i})[a,b]$ is a double lattice path, then we replace all non-vertical steps between lines $x=b_i$ and $x=b_i+1$, together with the vertical steps on lines $x=b_i$ and $x=b_i+1$ of $\tilde{P}(u_{s_i},v_{s_i})$ by the ones of $\tilde{P}_{b_i}(u_{s_i},v_{s_i})[a,b]$, and we set $a_i=a$.
\item Otherwise, $\tilde{P}_{b_i}(u_{t_i},v_{t_i})[b,a]$ is a double lattice path, then we replace all non-vertical steps between lines $x=b_i$ and $x=b_i+1$, together with the vertical steps on lines $x=b_i$ and $x=b_i+1$ of $\tilde{P}(u_{t_i},v_{t_i})$ by the ones of $\tilde{P}_{b_i}(u_{t_i},v_{t_i})[b,a]$, and we set $a_i=b$.
\end{enumerate}
In this way we retrieve the non-crossing $g$-tuple $P$ of double lattice paths as well as a sequence $\{a_i\}_{i=1}^r$ of positive integers, so that the $i$-th touchpoint $(b_i+1,d_i)$ of $P$ corresponds to the $\CMcal{C}$-pair $([y_1,y_2],b_i)$ of $\tilde{P}$ for every $i$. In fact, $(3)$-$(4)$ is the inverse process of $(1)$-$(2)$.

We notice that for any $\pi\in S_g$, any separable $g$-tuple (\ref{E:good3})
of double lattice paths, after the above process, yields a non-crossing $g$-tuple of double lattice paths. Because of Lemma~\ref{L:compatible}, $\pi$ in (\ref{E:good3}) must be the identity permutation and the proof is complete.
\end{proof}
%%%%%%%%%%%%%%%%%%%%%%%%%%%%%%%%%%%%%%
\begin{example}
Consider the outside nested decomposition $\Phi=(\Theta_1,\Theta_2,\Theta_3)$ in Figure~\ref{F:pic2-01} where $\Phi$ has three common special corners $(4,1),(3,3),(2,5)$. Given a pair $(\{a_i\}_{i=1}^3,P)$ where $(a_1,a_2,a_3)=(5,3,5)$ and $P$ is a non-crossing triple of double lattice paths given in Figure~\ref{F:pic2-6}. The corresponding separable triple $\tilde{P}$ of double lattice paths is shown in Figure~\ref{F:pic3-4}.

For instance, when we insert $a_1=5$ to the triple $P$ of double lattice paths in Figure~\ref{F:pic2-6}. Since the $(-3)$-point of $P(u_2,v_2)$ is above the one of $P(u_3,v_3)$, and $P_{-3}(u_2,v_2)[3,5]$ is not a double lattice path because a down-vertical step on line $x=-2$ precedes a horizontal step; see condition $(3)$ of Definition~\ref{D:plattice} and Figure~\ref{F:f2}. So between and on lines $x=-3$ and $x=-2$, $\tilde{P}(u_3,v_3)$ has the same steps as in $P_{-3}(u_3,v_3)[3,5]$.
\end{example}
%%%%%%%%%%%%%%%%%%%%%%%%%%%%%%%%%%%%%%%%%%
\begin{figure}[htbp]
\begin{center}
\includegraphics[scale=1.0]{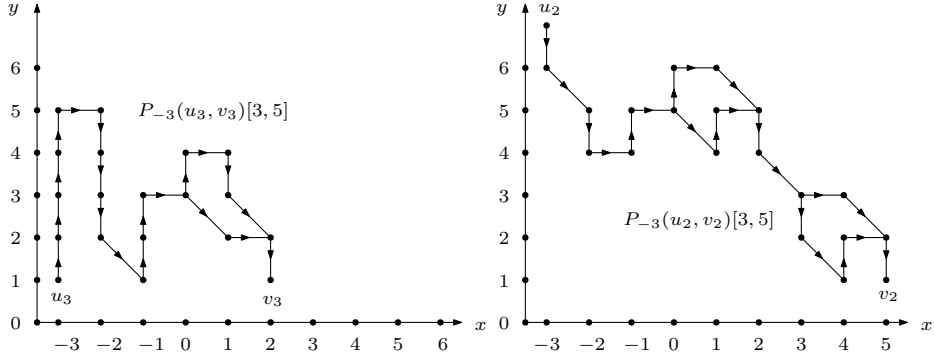}
\caption{$P_{-3}(u_2,v_2)[3,5]$ is not a double lattice path, while $P_{-3}(u_3,v_3)[3,5]$ is a double lattice path.
\label{F:f2}}
\end{center}
\end{figure}
%%%%%%%%%%%%%%%%%%%%%%%%%%%%%%%%%%%%%%%%%%
\subsection{Construct the involution}\label{ss:invo}
For any permutation $\pi=\pi_1\pi_2\cdots\pi_g\in S_g$, the {\em inversion} of $\pi$ is
$\mbox{inv}(\pi)=\vert\{(i,j):\pi_i>\pi_j,\,i<j\}\vert$ and we may interpret the determinant in Theorem~\ref{T:thick} as
\begin{align}\label{E:gen}
\det[s_{\Theta_i\#\Theta_j}(X)]_{i=1}^g=\sum_{\pi\in S_g}(-1)^{\scriptsize{\mbox{inv}}(\pi)}\prod_{i=1}^g
s_{\Theta_i\#\Theta_{\pi_i}}(X).
\end{align}
By Lemma~\ref{L:bij3} we know that all semistandard Young tableaux from $\CMcal{T}_{\Theta_i\#\Theta_{j}}$ are in bijection with all double lattice paths in $\CMcal{P}(u_{j},v_i)$. It follows that all $g$-tuples
\begin{align*}
(T_{\Theta_1\#\Theta_{\pi_1}},T_{\Theta_2\#\Theta_{\pi_2}},\ldots, T_{\Theta_g\#\Theta_{\pi_g}})
\end{align*}
where $T_{\Theta_i\#\Theta_{\pi_i}}$ is a semistandard Young tableau of thickened shape $\Theta_i\#\Theta_{\pi_i}$, are in bijection with all $g$-tuples (\ref{E:good3}) of double lattice paths where $\tilde{P}(u_{\pi_i},v_i)=f(T_{\Theta_i\#\Theta_{\pi_i}})$.

For every double lattice path $\tilde{P}(u_{\pi_i},v_i)$ in (\ref{E:good3}), if two non-vertical steps end at the same point $(a,b)$, we assign these two steps with a single weight $x_b$. For every other horizontal step or diagonal step, we assign each step with a weight $x_b$ if the step ends at $(a,b)$. For every vertical step, we assign it with weight $1$. Furthermore, the weight of every double lattice path $\tilde{P}(u_{\pi_i},v_i)$, denoted by $\CMcal{W}(\tilde{P}(u_{\pi_i},v_i))$, is the product of all weights on the steps of $\tilde{P}(u_{\pi_i},v_i)$ and we use $q[u_{\pi_i},v_i](X)$ to denote the generating function of all double lattice paths in $\CMcal{P}(u_{\pi_i},v_i)$, that is, $q[u_{\pi_i},v_i](X)$ is the sum of all weighted double lattice paths from $\CMcal{P}(u_{\pi_i},v_i)$. The relation between these notations is
\begin{align}\label{E:re}
\sum_{\tilde{P}}\prod_{i=1}^g \CMcal{W}(\tilde{P}(u_{\pi_i},v_i))
=\prod_{i=1}^g \sum_{\CMcal{P}}\CMcal{W}(\tilde{P}(u_{\pi_i},v_i))
=\prod_{i=1}^g q[u_{\pi_i},v_i](X)
\end{align}
where the first sum runs over all $g$-tuples (\ref{E:good3}) of double lattice paths and the second sum runs over all double lattice paths $\tilde{P}(u_{\pi_i},v_i)$ from the set $\CMcal{P}(u_{\pi_i},v_i)$. We recall that Lemma~\ref{L:bij3} implies $q[u_{\pi_i},v_i](X)=s_{\Theta_{i}\#\Theta_{\pi_i}}(X)$. If $\Theta_i\#\Theta_{\pi_i}=\varnothing$, then $s_{\varnothing}(X)=q[u_{\pi_i},v_i](X)=1$ since the only double lattice path from $u_{\pi_i}$ to $v_i$ has no non-vertical steps.
If $\Theta_i\#\Theta_{\pi_i}$ is undefined, then $s_{\Theta_i\#\Theta_{\pi_i}}(X)=q[u_{\pi_i},v_i](X)=0$ since the set $\CMcal{P}(u_{\pi_i},v_i)$ is undefined. Together with (\ref{E:gen}) and (\ref{E:re}), one obtains
\begin{align}\label{E:gen1}
\det[s_{\Theta_i\#\Theta_j}(X)]_{i=1}^g&=
\sum_{\pi\in S_g}\sum_{\tilde{P}}(-1)^{\scriptsize{\mbox{inv}}(\pi)}\prod_{i=1}^g \CMcal{W}(\tilde{P}(u_{\pi_i},v_i))
\end{align}
which can be viewed as a generating function for all pairs $(\pi,\tilde{P})$ where $\pi\in S_g$ and $\tilde{P}$ is any $g$-tuple (\ref{E:good3}) of double lattice paths. From Proposition~\ref{L:bij5} and Proposition~\ref{P:nonspr}, it follows that the generating function for all pairs $(\mbox{id},\tilde{P})$ when $\tilde{P}$ is separable, equals the generating function for all pairs $(\{a_i\}_{i=1}^r,P)$ where $P$ is a non-crossing $g$-tuple (\ref{E:good2}) of double lattice paths, that is,
%%%%%%%%%%%%%%%%%%%%%%%%%%%%%%%%%%%%%%%%
\begin{align}\label{E:nonspe}
(\sum_{i=1}^{\infty}x_i)^r s_{\lambda/\mu}(X)=p_{1^r}(X)s_{\lambda/\mu}(X).
\end{align}
%%%%%%%%%%%%%%%%%%%%%%%%%%%%%%%%%%%%%%%%
So in order to prove Theorem~\ref{T:thick}, it remains to find an involution on all pairs $(\pi,\tilde{P})$ when $\tilde{P}$ is non-separable. From Definition~\ref{D:sdlp} it is clear that a $g$-tuple $\tilde{P}$ (see (\ref{E:good3})) of double lattice paths is non-separable if and only if $\tilde{P}$ is non-separable at a point or at a $\CMcal{C}$-pair. It should be noted that there is {\em no} common integer $c$ such that $\tilde{P}$ is non-separable at point $(c,y)$ and at a $\CMcal{C}$-pair $([y_1,y_2],c)$ according to Definition~\ref{D:spe1} and \ref{D:rep}. So we consider the minimal integer $c^*$ such that
\begin{itemize}
\item $\tilde{P}$ is non-separable at a point on line $x=c^*$ or is non-separable at a $\CMcal{C}$-pair $([y_1,y_2],c^*)$ for some $y_1,y_2$;
\item $\tilde{P}$ is neither non-separable at any point $(c,y)$ nor at any $\CMcal{C}$-pair $([y_1,y_2],\tilde{c})$ when $c,\tilde{c}<c^*$.
\end{itemize}
We choose a {\em minimum} of any non-separable $g$-tuple $\tilde{P}$ (see (\ref{E:good3})) of double lattice paths to be
\begin{enumerate}
\item the point $(c^*,y)$ if it is the first point on line $x=c^*$ from top to bottom such that $\tilde{P}$ is non-separable at the point $(c^*,y)$;
\item the $\CMcal{C}$-pair $([y_1,y_2],c^*)$ if $(c^*,y_2),(c^*,y_1)$ are the first two $c^*$-points on line $x=c^*$ from top to bottom such that $\tilde{P}$ is non-separable at the $\CMcal{C}$-pair $([y_1,y_2],c^*)$.
\end{enumerate}
We are now ready to construct the involution $\omega$ on all non-separable $g$-tuples $\tilde{P}$ (see (\ref{E:good3})) of double lattice paths by distinguishing the cases when the minimum of $\tilde{P}$ is a single point $(c^*,y)$ or a $\CMcal{C}$-pair $([y_1,y_2],c^*)$. For each case, we will express the involution $\omega$ as
\begin{align*}
(\pi,\tilde{P})\mapsto (\sigma,P^*)=\omega((\pi,\tilde{P}))
\end{align*}
where $\pi,\sigma\in S_g$ and $\tilde{P},P^*$ are two non-separable $g$-tuples of double lattice paths with
\begin{align}\label{E:Qspc}
P^*=(P^*(u_{\sigma_1},v_1),P^*(u_{\sigma_2},v_2),\ldots,P^*(u_{\sigma_g},v_g)).
\end{align}
For each case below, the involution $\omega$ has the following properties:
%%%%%%%%%%%%%%%%%%%%%%%%%%%%%%%%%%%%%%%%%%%%
\begin{enumerate}
\item $\omega$ is weight-preserving, that is,
\begin{align*}
\prod_{q=1}^g\CMcal{W}(P^*(u_{\sigma_q},v_q))
    =\prod_{q=1}^g\CMcal{W}(\tilde{P}(u_{\pi_q},v_q));
\end{align*}
\item $\omega$ is sign-reversing, that is, \mbox{inv}$(\pi)=\mbox{inv}(\sigma)\pm 1$;
\item $\omega$ is closed, that is, $\tilde{P}$ and $P^*$ belong to the same case.
\end{enumerate}
%%%%%%%%%%%%%%%%%%%%%%%%%%%%%%%%%%%%%%%%%%%%
{\em Case $1$}: if the minimum of $\tilde{P}$ is the point $(c^*,y)$, and among all double lattice paths that are passing the point $(c^*,y)$, assume that $\tilde{P}(u_{\pi_i},v_i)$ and $\tilde{P}(u_{\pi_j},v_j)$ of $\tilde{P}$ are two double lattice paths whose indices $i$ and $j$ are the smallest and the second smallest. Since neither $c^*$ nor $c^*-1$ is the content of some common special corner of $\Phi$, all steps of $\tilde{P}$ between lines $x=c^*$ and $x=c^*+1$ are all horizontal steps or all diagonal steps. By our choice of $c^*$, all steps of $\tilde{P}(u_{\pi_i},v_i)$ between lines $x=c^*-1$ and $x=c^*$ are disjoint with the ones of $\tilde{P}(u_{\pi_j},v_j)$.

Using the notations $\tilde{P}(u_{\pi_i},v)$ and $\tilde{P}(v,v_i)$ to denote the segments of the double lattice path $\tilde{P}(u_{\pi_i},v_i)$ from $u_{\pi_i}$ to the point $v=(c^*,y)$ and from the point $v=(c^*,y)$ to $v_i$ (similarly for $\tilde{P}(u_{\pi_j},v_j)$), we may define the pair $(\sigma,P^*)=\omega((\pi,\tilde{P}))$ where $\sigma=\pi\circ(i \,j)$ as follows. %and
%\begin{align*}
%Q=(Q(u_{\sigma_1},v_1),Q(u_{\sigma_2},v_2),\ldots,Q(u_{\sigma_m},v_m)).
%\end{align*}
For $q\ne i$, $q\ne j$, we set $P^*(u_{\sigma_q},v_q)=\tilde{P}(u_{\pi_q},v_q)$ and
\begin{align*}
P^*(u_{\sigma_i},v_i)=\tilde{P}(u_{\pi_j},v)\tilde{P}(v,v_i),\,\,
P^*(u_{\sigma_j},v_j)=\tilde{P}(u_{\pi_i},v)\tilde{P}(v,v_j).
\end{align*}
We will show that $P^*(u_{\sigma_i},v_i)$ is a double lattice path from $u_{\sigma_i}=u_{\pi_j}$ to $v_i$ and $P^*(u_{\sigma_j},v_j)$ is a double lattice path from $u_{\pi_i}=u_{\sigma_j}$ to $v_j$ by discussing the ending points of the non-vertical steps between lines $x=c^*-1$ and $x=c^*+1$.

Here, without loss of generality, we assume that the steps between lines $x=c^*-1$ and $x=c^*$ are horizontal, while the steps between lines $x=c^*$ and $x=c^*+1$ are diagonal. Suppose that the ending points of non-vertical steps from $\tilde{P}(u_{\pi_i},v_i)$ are the points $(c^*,a_1)$ and $(c^*+1,a_2)$ where $a_1>a_2$, and the ones from $\tilde{P}(u_{\pi_j},v_j)$ are the points $(c^*,y)$ and $(c^*+1,b_2)$ where $y>b_2$; see Figure~\ref{F:c1}. Since $\tilde{P}(u_{\pi_i},v_i)$ and $\tilde{P}(u_{\pi_j},v_j)$ are intersecting at the point $(c^*,y)$, one has $a_2<y<a_1$, which implies $b_2<y<a_1$. So there is no single up-vertical step on line $x=c^*$ that is preceding the diagonal step in $P^*(u_{\sigma_i},v_i)$ or $P^*(u_{\sigma_j},v_j)$. This indicates that $P^*(u_{\sigma_i},v_i)$ and $P^*(u_{\sigma_j},v_j)$ are double lattice paths according to (4) in Definition~\ref{D:plattice}.

Furthermore, $\omega$ is closed within all non-separable $g$-tuples of double lattice paths that belong to case $1$, because by construction $\tilde{P}$ and $P^*$ are non-separable at the same points and the same $\CMcal{C}$-pairs. In particular, the minimum of $P^*$ is also the point $(c^*,y)$. See Figure~\ref{F:c1}.
%%%%%%%%%%%%%%%%%%%%%%%%%%%%%%%%%%%%%%%%%%%%%%%%
\begin{figure}[htbp]
\begin{center}
\includegraphics[scale=1.0]{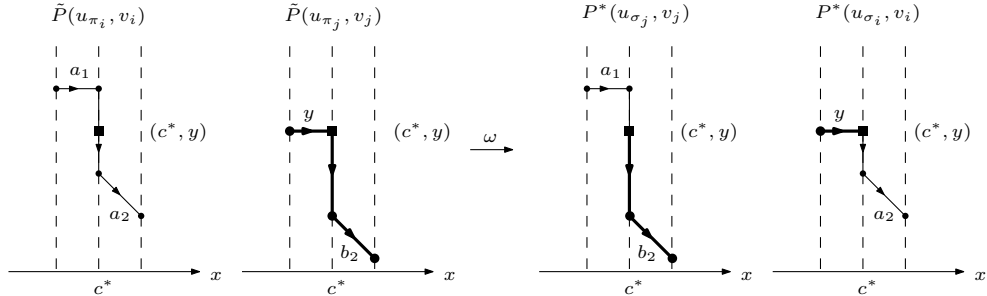}
\caption{The involution $\omega$ for case $1$ when the minimum of $\tilde{P}$ is a point $(c^*,y)$ (marked by a black square) and all integers represent the $y$-th coordinates of all ending points from the non-vertical steps.
\label{F:c1}}
\end{center}
\end{figure}
%%%%%%%%%%%%%%%%%%%%%%%%%%%%%%%%%%%%%%%%%%%%%%%%

{\em Case $2$}: if the minimum of $\tilde{P}$ is the $\CMcal{C}$-pair $([y_1,y_2],c^*)$, we assume that $i,j$ are the smallest indices of $s,t$ satisfying $\tilde{P}(u_{\pi_s},v_s),\tilde{P}(u_{\pi_t},v_t)$ are non-separable at the $\CMcal{C}$-pair $([y_1,y_2],c^*)$. By our choice of $c^*$, all steps of $\tilde{P}(u_{\pi_i},v_i)$ between lines $x=c^*-1$ and $x=c^*$ are disjoint with the ones of $\tilde{P}(u_{\pi_j},v_j)$. Consequently, suppose that between lines $x=c^*$ and $x=c^*+2$, the diagonal step and the horizontal step of $\tilde{P}(u_{\pi_j},v_j)$ end at
\begin{align*}
(c^*+1,b)\,\mbox{ and }\, (c^*+2,y_4),
\end{align*}
and the horizontal step and the diagonal step of $\tilde{P}(u_{\pi_i},v_i)$ end at
\begin{align*}
(c^*+1,a)\,\mbox{ and }\, (c^*+2,y_3).
\end{align*}
Furthermore, if there is a horizontal step and a diagonal step of $\tilde{P}(u_{\pi_j},v_j)$ ending on line $x=c^*+1$ and on line $x=c^*+2$, we assume they end at
\begin{align*}
(c^*+1,d_2)\,\mbox{ and }\, (c^*+2,y_4).
\end{align*}
It should be mentioned that such horizontal step and diagonal step are not contained in $\tilde{P}(u_{\pi_j},v_j)$ if the starting box or the ending box of $\Theta_j$ is the common special corner of $\Theta_i$ and $\Theta_j$. But for this situation the discussion on the involution $\omega$ follows analogously, so we focus on the case when the horizontal step ending at $(c^*+1,d_2)$ and the diagonal step $(c^*+2,y_4)$ are contained in $\tilde{P}(u_{\pi_j},v_j)$. Likewise, if there is a diagonal step and a horizontal step of $P(u_{\pi_i},v_i)$ ending on line $x=c^*+1$ and on line $x=c^*+2$, we assume that they end at
\begin{align*}
(c^*+1,d_1)\,\mbox{ and }\, (c^*+2,y_3).
\end{align*}
See Figure~\ref{F:c0} where all integers represent the $y$-th coordinates of all ending points from the non-vertical steps.
%%%%%%%%%%%%%%%%%%%%%%%%%%%%%%%%%%%%%%%%%%%%%%%%
\begin{figure}[htbp]
\begin{center}
\includegraphics[scale=1.0]{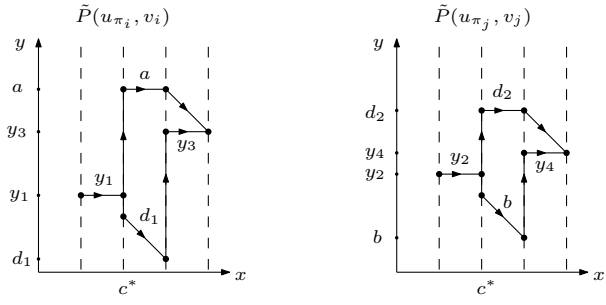}
\caption{The steps of $\tilde{P}(u_{\pi_i},v_i)$ and $\tilde{P}(u_{\pi_j},v_j)$ between lines $x=c^*-1$ and $x=c^*+2$ if the minimum of $\tilde{P}$ is the $\CMcal{C}$-pair $([y_1,y_2],c^*)$.
\label{F:c0}}
\end{center}
\end{figure}
%%%%%%%%%%%%%%%%%%%%%%%%%%%%%%%%%%%%%%%%%%%%%%%%

Since $\tilde{P}(u_{\pi_i},v_i)$ and $\tilde{P}(u_{\pi_j},v_j)$ are double lattice paths, from Definition~\ref{D:plattice} we find that $d_1<y_1\le a$, $d_1\le y_3<a$, $b<y_2\le d_2$ and $b\le y_4<d_2$. Since neither $\tilde{P}_{c^*}(u_{\pi_i},v_i)[a,b]$ nor $\tilde{P}_{c^*}(u_{\pi_j},v_j)[b,a]$ is a double lattice path, then the integer $b$ satisfies $b<y_1$ or $b\le y_3$ and the integer $a$ satisfies $a\ge y_2$ or $a>y_4$. So under the assumption $y_1<y_2$, we shall consider the following disjoint sub-cases:
\begin{itemize}
\item[] case $2.1$: $b<y_1<y_2\le a$;
\item[] case $2.2$: $b<y_1$ and $y_4<a<y_2$;
\item[] case $2.3$: $y_1\le b\le y_3$ and $a\ge y_2$;
\item[] case $2.4$: $y_1\le b\le y_3$ and $y_4<a<y_2$.
\end{itemize}

{\em Case $2.1$}: if $b<y_1<y_2\le a$, we use the notations $\tilde{P}(u_{\pi_i},x\vert_{c^*})$ and $\tilde{P}(x\vert_{c^*},v_i)$ to denote the segments of the double lattice path $\tilde{P}(u_{\pi_i},v_i)$ from $u_{\pi_i}$ to all non-vertical steps ending on line $x=c^*$ and from all non-vertical steps of $\tilde{P}(u_{\pi_i},v_i)$ starting on line $x=c^*$ to $v_i$ (similarly for $\tilde{P}(u_{\pi_j},v_j)$). Furthermore, $\tilde{P}(u_{\pi_i},x\vert_{c^*})\tilde{P}(x\vert_{c^*},v_j)$ is obtained by connecting two segments $\tilde{P}(u_{\pi_i},x\vert_{c^*})$ and $\tilde{P}(x\vert_{c^*},v_j)$ with new vertical steps on line $x=c^*$. Here we may define the pair $(\sigma,P^*)=\omega((\pi,\tilde{P}))$ as follows. For $q\ne i$, $q\ne j$, we set $P^*(u_{\sigma_q},v_q)=\tilde{P}(u_{\pi_q},v_q)$ and
\begin{align*}
P^*(u_{\sigma_j},v_j)=\tilde{P}(u_{\pi_i},x\vert_{c^*})\tilde{P}(x\vert_{c^*},v_j),\,
P^*(u_{\sigma_i},v_i)=\tilde{P}(u_{\pi_j},x\vert_{c^*})\tilde{P}(x\vert_{c^*},v_i)
\end{align*}
where $P^*(u_{\sigma_j},v_j)$ is a double lattice path from $u_{\sigma_j}=u_{\pi_i}$ to $v_j$ and $P^*(u_{\sigma_i},v_i)$ is a double lattice path from $u_{\pi_j}=u_{\sigma_i}$ to $v_i$. This is guaranteed by the relations $d_1<y_2\le a$ and $d_2>y_1>b$. Furthermore, $\omega$ is closed within all non-separable $g$-tuples of double lattice paths that belong to case $2.1$, because $P^*$ also belongs to case $2.1$ since $d_1<y_1<y_2\le d_2$, and by construction $\tilde{P},P^*$ are non-separable at the same points and at the same $\CMcal{C}$-pairs. In particular, the minimum of $P^*$ is also the $\CMcal{C}$-pair $([y_1,y_2],c^*)$. See Figure~\ref{F:c2} for an example of case $2.1$.
%%%%%%%%%%%%%%%%%%%%%%%%%%%%%%%%%%%%%%%%%%%%%%%%
\begin{figure}[htbp]
\begin{center}
\includegraphics[scale=1.0]{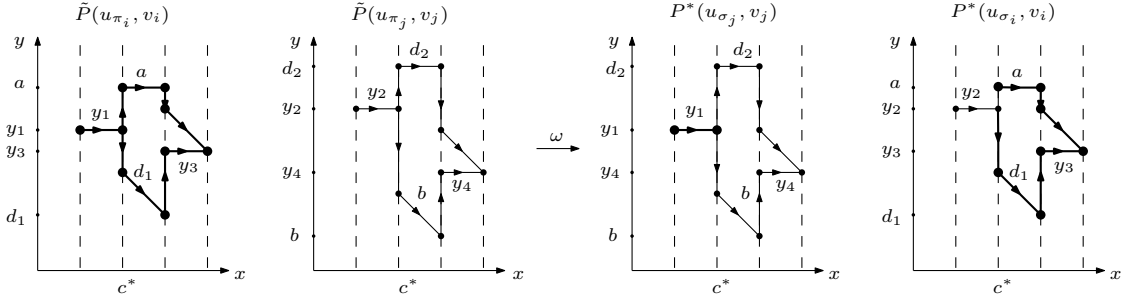}
\caption{The involution $\omega$ for case $2.1$ when the minimal of $\tilde{P}$ is the $\CMcal{C}$-pair $([y_1,y_2],c^*)$ and $b<y_1<y_2\le a$. \label{F:c2}}
\end{center}
\end{figure}
%%%%%%%%%%%%%%%%%%%%%%%%%%%%%%%%%%%%%%%%%%%%%%%%

{\em Case $2.2$}: if $b<y_1$ and $y_4<a<y_2$, we use the notation $\tilde{P}(u_{\pi_i},x\vert_{c^*}^+)$ to denote the segment of the double lattice path $\tilde{P}(u_{\pi_i},v_i)$ from $u_{\pi_i}$ to its $c^*$-point, as well as the step from the $c^*$-point to the ending point of the horizontal step between lines $x=c^*$ and $x=c^*+1$. We use the notation $\tilde{P}(x\vert_{c^*}^-,v_i)$ to denote the segment of the double lattice path $\tilde{P}(u_{\pi_i},v_i)$ that is complement to the segment $\tilde{P}(u_{\pi_i},x\vert_{c^*}^+)$ except the down-vertical steps on line $x=c^*$ and $x=c^*+1$; similarly for $\tilde{P}(u_{\pi_j},v_j)$. Furthermore, $\tilde{P}(u_{\pi_i},x\vert_{c^*}^+)\tilde{P}(x\vert_{c^*}^-,v_j)$ is obtained by connecting two segments $\tilde{P}(u_{\pi_i},x\vert_{c^*}^+)$ and $\tilde{P}(x\vert_{c^*}^-,v_j)$ by new down-vertical steps on lines $x=c^*$ and $x=c^*+1$. Here we may define the pair $(\sigma,P^*)=\omega((\pi,\tilde{P}))$ where $\sigma=\pi\circ(i\,j)$ as follows. For $q\ne i$, $q\ne j$, we set $P^*(u_{\sigma_q},v_q)=\tilde{P}(u_{\pi_q},v_q)$ and
\begin{align*}
P^*(u_{\sigma_j},v_j)=\tilde{P}(u_{\pi_i},x\vert_{c^*}^+)\tilde{P}(x\vert_{c^*}^-,v_j),\,
P^*(u_{\sigma_i},v_i)=\tilde{P}(u_{\pi_j},x\vert_{c^*}^+)\tilde{P}(x\vert_{c^*}^-,v_i)
\end{align*}
where $P^*(u_{\sigma_j},v_j)$ is a double lattice path from $u_{\sigma_j}=u_{\pi_i}$ to $v_j$ and $P^*(u_{\sigma_i},v_i)$ is a double lattice path from $u_{\pi_j}=u_{\sigma_i}$ to $v_i$. This is guaranteed by the assumption $b<y_1$ and $y_4<a<y_2$. To be precise, $d_1<y_2$ holds because $d_1<y_1$ and $y_1<y_2$; $y_3<d_2$ holds because $y_3<a$ and $a<y_2<d_2$; $b<y_1$ and $a>y_4$ hold because of the assumption. Furthermore, $\omega$ is closed within all non-separable $g$-tuples of double lattice paths that belong to case $2.2$, that is, $P^*$ also belongs to case $2.2$ since $d_1<y_1$ and $y_3<a<y_2$, and by construction $\tilde{P},P^*$ are non-separable at the same points and the same $\CMcal{C}$-pairs. In particular, the minimum of $P^*$ is also the $\CMcal{C}$-pair $([y_1,y_2],c^*)$. See Figure~\ref{F:c3} for an example of case $2.2$.
%%%%%%%%%%%%%%%%%%%%%%%%%%%%%%%%%%%%%%%%%%%%%%%%
\begin{figure}[htbp]
\begin{center}
\includegraphics[scale=1.0]{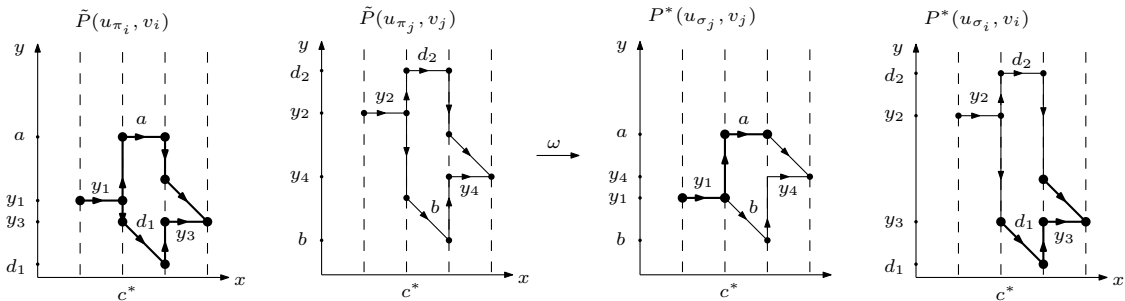}
\caption{The involution $\omega$ for case $2.2$ when the minimal of $\tilde{P}$ is the $\CMcal{C}$-pair $([y_1,y_2],c^*)$ such that $b<y_1$ and $y_4<a<y_2$.
\label{F:c3}}
\end{center}
\end{figure}
%%%%%%%%%%%%%%%%%%%%%%%%%%%%%%%%%%%%%%%%%%%%%%%%

{\em Case $2.3$}: if $y_1\le b\le y_3$ and $a\ge y_2$, we use the notation $\tilde{P}(u_{\pi_i},x\vert_{c^*}^-)$ to denote the segment of the double lattice path $\tilde{P}(u_{\pi_i},v_i)$ from $u_{\pi_i}$ to its $c^*$-point, together with the step from the $c^*$-point to the ending point of the diagonal step between lines $x=c^*$ and $x=c^*+1$. We use the notation $\tilde{P}(x\vert_{c^*}^+,v_i)$ to denote the segment of the double lattice path $\tilde{P}(u_{\pi_i},v_i)$ that is complement to the segment $\tilde{P}(u_{\pi_i},x\vert_{c^*}^-)$ except the up-vertical steps on lines $x=c^*$ and $x=c^*+1$; similarly for $\tilde{P}(u_{\pi_j},v_j)$. Furthermore, $\tilde{P}(u_{\pi_i},x\vert_{c^*}^-)\tilde{P}(x\vert_{c^*}^+,v_j)$ is obtained by connecting two segments $\tilde{P}(u_{\pi_i},x\vert_{c^*}^-)$ and $\tilde{P}(x\vert_{c^*}^+,v_j)$ by new up-vertical steps on lines $x=c^*$ and $x=c^*+1$. Here we may define the pair $(\sigma,P^*)=\omega((\pi,\tilde{P}))$ where $\sigma=\pi\circ(i\,j)$ as follows.
For $q\ne i$, $q\ne j$, we set $P^*(u_{\sigma_q},v_q)=\tilde{P}(u_{\pi_q},v_q)$ and
\begin{align*}
P^*(u_{\sigma_j},v_j)=\tilde{P}(u_{\pi_i},x\vert_{c^*}^-)\tilde{P}(x\vert_{c^*}^+,v_j),\,
P^*(u_{\sigma_i},v_i)=\tilde{P}(u_{\pi_j},x\vert_{c^*}^-)\tilde{P}(x\vert_{c^*}^+,v_i)
\end{align*}
where $P^*(u_{\sigma_j},v_j)$ is a double lattice path from $u_{\sigma_j}=u_{\pi_i}$ to $v_j$ and $P^*(u_{\sigma_i},v_i)$ is a double lattice path from $u_{\pi_j}=u_{\sigma_i}$ to $v_i$. This is guaranteed by the assumption $y_1\le b\le y_3$ and $a\ge y_2$. To be precise, $d_2>y_1$ holds because $d_2>y_2$ and $y_2>y_1$; $d_1\le y_4$ holds because $d_1<y_1\le b\le y_4$; $a\ge y_2$ and $b\le y_3$ hold because of the assumption. See Figure~\ref{F:c4} for an example of case $2.3$.
%%%%%%%%%%%%%%%%%%%%%%%%%%%%%%%%%%%%%%%%%%%%%%%%
\begin{figure}[htbp]
\begin{center}
\includegraphics[scale=1.0]{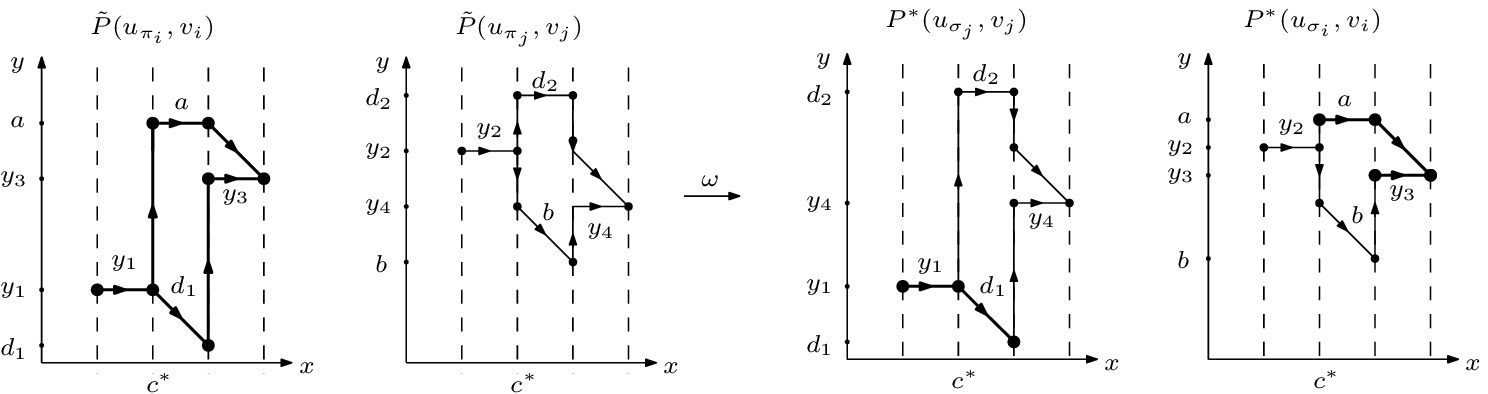}
\caption{The involution $\omega$ for case $2.3$ when the minimal of $\tilde{P}$ is the $\CMcal{C}$-pair $([y_1,y_2],c^*)$ such that $y_1\le b\le y_3$ and $a\ge y_2$.
\label{F:c4}}
\end{center}
\end{figure}
%%%%%%%%%%%%%%%%%%%%%%%%%%%%%%%%%%%%%%%%%%%%%%%%
Furthermore, $\omega$ is closed within all non-separable $g$-tuples of double lattice paths that belong to case $2.3$, because $P^*$ also belongs to case $2.3$ since $y_1\le b\le y_4$ and $d_2\ge y_2$, and by construction $\tilde{P},P^*$ are non-separable at the same points and the same $\CMcal{C}$-pairs. In particular, the minimum of $P^*$ is also the $\CMcal{C}$-pair $([y_1,y_2],c^*)$. See Figure~\ref{F:c4} for an example of case $2.3$.

{\em Case $2.4$}: if $y_1\le b\le y_3$ and $y_4<a<y_2$, then we may define the pair $(\sigma,P^*)=\omega((\pi,\tilde{P}))$ where $\sigma=\pi\circ(i\,j)$ as follows. For $q\ne i$, $q\ne j$, we set $P^*(u_{\sigma_q},v_q)=\tilde{P}(u_{\pi_q},v_q)$ and
\begin{align*}
P^*(u_{\sigma_j},v_j)=\tilde{P}(u_{\pi_i},x\vert_{c^*+1})\tilde{P}(x\vert_{c^*+1},v_j),\,
P^*(u_{\sigma_i},v_i)=\tilde{P}(u_{\pi_j},x\vert_{c^*+1})\tilde{P}(x\vert_{c^*+1},v_i)
\end{align*}
where $P^*(u_{\sigma_j},v_j)$ is a double lattice path from $u_{\sigma_j}=u_{\pi_i}$ to $v_j$ and $P^*(u_{\sigma_i},v_i)$ is a double lattice path from $u_{\pi_j}=u_{\sigma_i}$ to $v_i$. This is guaranteed by the assumption $y_1\le b\le y_3$ and $y_4<a<y_2$. To be precise, $d_1<y_4<a$ holds because $d_1<y_1\le b\le y_4<a$; and $b\le y_3<d_2$ holds because $b\le y_3<a<y_2\le d_2$.
Furthermore, $\omega$ is closed within all non-separable $g$-tuples of double lattice paths that belong to case $2.4$, because $P^*$ also belongs to case $2.4$ since $y_1\le b\le y_4$ and $y_3<a<y_2$, and by construction $\tilde{P},P^*$ are non-separable at the same points and the same $\CMcal{C}$-pairs. In particular, the minimum of $P^*$ is also the $\CMcal{C}$-pair $([y_1,y_2],c^*)$. See Figure~\ref{F:c5} for an example of case $2.4$.
%%%%%%%%%%%%%%%%%%%%%%%%%%%%%%%%%%%%%%%%%%%%%%%%
\begin{figure}[htbp]
\begin{center}
\includegraphics[scale=1.0]{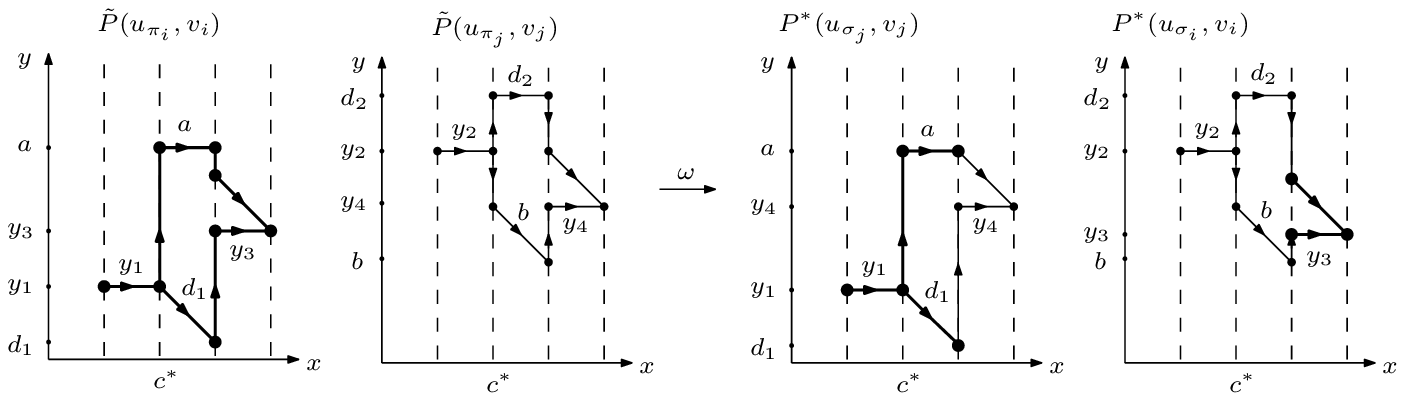}
\caption{The involution $\omega$ for case $2.4$ when the minimal of $\tilde{P}$ is the $\CMcal{C}$-pair $([y_1,y_2],c^*)$ such that $y_1\le b\le y_3$ and $y_4<a<y_2$.
\label{F:c5}}
\end{center}
\end{figure}
%%%%%%%%%%%%%%%%%%%%%%%%%%%%%%%%%%%%%%%%%%%%%%%%

For each case (case $1$ or case $2.1$-$2.4$), it is clear that $(\pi,\tilde{P})\mapsto(\sigma,P^*)=\omega((\pi,\tilde{P}))$ is an involution which preserves the weight of the double lattice path and changes the inversion of the permutation by $1$.

{\em Proof of Theorem~\ref{T:thick}}.
From (\ref{E:gen1}) and the involution $(\pi,\tilde{P})\mapsto (\sigma,P^*)$ in Subsection~\ref{ss:invo}, we find that only the generating function for all pairs $(\mbox{id},\tilde{P})$ where $\tilde{P}$ is any separable $g$-tuple of double lattice paths, is remained on the right hand side of (\ref{E:gen1}). In combination of (\ref{E:nonspe}), (\ref{E:HG}) follows immediately.
\qed
%\begin{remark}
%The facts in Remark~\ref{R:1} are of particular importance for the construction of the involution $\omega$ in subsection~\ref{ss:invo} because we need to have that, no matter where the minimum of (\ref{E:good3}) is located, two double lattice paths $\tilde{P}(u_{\pi_i},v_i)$ and $\tilde{P}(u_{\pi_j},v_j)$ after the involution $\omega$ yield two uniform double lattice paths $P^*(u_{\sigma_i},v_i)$ and $P^*(u_{\sigma_j},v_j)$, that is, all non-vertical steps of any double lattice path in $\CMcal{P}(u_j,v_i)$ with respect to $\Phi$ are determined once the starting point $u_j$ and the ending point $v_i$ are fixed.
%\end{remark}

{\em Proof of Corollary~\ref{C:thick2}}. We refer the readers to Chapter $7$ of \cite{Stanley:ec2} for a full description of the exponential specialization. Let $[x_1x_2\cdots x_n]f$ denote the coefficient of $x_1x_2\cdots x_n$ in $f$, the exponential specialization ex of the symmetric function $f$ is defined as
\begin{align*}
\mbox{ex}(f)=\sum_{n\ge 0}[x_1x_2\cdots x_n]f\frac{t^n}{n!}
\end{align*}
and $\mbox{ex}_1(f)=\mbox{ex}(f)_{t=1}$. Let $N=\vert \lambda/\mu\vert$ and $a_{i,j}=\vert\Theta_i\#\Theta_j\vert$, then one has \begin{align*}
\mbox{ex}(p_{1^r}(X))=t^r,\,\,
\mbox{ex}(s_{\lambda/\mu}(X))=f^{\lambda/\mu}(N!)^{-1}t^{N}\,\,
\mbox{ and }\,\,
\mbox{ex}(s_{\Theta_i\#\Theta_j}(X))=f^{\Theta_i\#\Theta_j}
(a_{i,j}!)^{-1}t^{a_{i,j}}.
\end{align*}
Consequently (\ref{E:HG1}) follows directly after we apply $\mbox{ex}_1$ on both sides of (\ref{E:HG}).
\qed
%%%%%%%%%%%%%%%%%%%%%%%%%%%%%%%%%%%%%%%%%%%%%%%%%%%%%%%%%%%%
\section{Application to the enumeration of $m$-strip tableaux}\label{E:app}
We will count the number of $m$-strip tableaux by applying Corollary~\ref{C:thick2}. It should be pointed out that the enumeration of $2k$-strip tableaux is a direct consequence of Theorem~\ref{T:HG}; see \cite{HG}. In \cite{MPP}, Morales, Pak and Panova also found that the enumeration of $2k$-strip tableaux can be simplified by applying Lascoux-Pragacz's theorem \cite{LP}, or more generally, Hamel and Goulden's theorem (Theorem~\ref{T:HG}).

\subsection{The $m$-strip tableaux}
Baryshnikov and Romik \cite{BR:07} counted the number of $m$-strip tableaux as a generalization of the classical formula from D. Andr\'{e} \cite{And:79} on the number of up-down permutations.
\begin{definition}[$m$-strip tableaux]
An {\em $m$-strip diagram} $\CMcal{D}_m(\tilde{\lambda};\tilde{\mu})$ contains three parts: head $\tilde{\lambda}$, tail $\tilde{\mu}$ and body. The body of an $m$-strip diagram consists of an elongated hexagonal shape with $n$ columns, where the numbers of boxes in the $n$ columns are
\begin{align*}
\lceil\frac{m+1}{2}\rceil,\lceil\frac{m+1}{2}\rceil+1,\ldots,m-1,m,m,\ldots,m,m-1,\ldots,
\lceil\frac{m+1}{2}\rceil+1, \lceil\frac{m+1}{2}\rceil.
\end{align*}
The first (resp. last) $\lfloor m/2\rfloor$ columns forms a standard diagram and the columns where each contains $m$ boxes forms a skew diagram of shape
\begin{align*}
(\underbrace{n-2\lfloor \frac{m}{2}\rfloor+2,\ldots,n-2\lfloor \frac{m}{2}\rfloor+2}_{\scriptsize{\mbox{length}}:\,m},n-2\lfloor \frac{m}{2}\rfloor+1,\ldots,2,1)/(n-2\lfloor \frac{m}{2}\rfloor+1,\ldots,2,1).
\end{align*}
The {\em head} $\tilde{\lambda}$ and {\em tail} $\tilde{\mu}$ are standard diagrams of length at most $\lfloor m/2\rfloor$ that are rotated and connected to the body by leaning against the sides of the body. The empty partition $(0)$ is always denoted by $\varnothing$ and {\em an $m$-strip tableau} is a standard Young tableau of the $m$-strip shape.
\end{definition}
\begin{remark}
Our definition of $m$-strip diagrams is slightly different to the one in \cite{BR:07} because \cite{BR:07} contains a minor typo on the number of boxes in the leftmost and the rightmost columns of any $m$-strip diagram. Our notation $\CMcal{D}_m(\tilde{\lambda};\tilde{\mu})$ is the notation $D$ in \cite{BR:07} and we find it more convenient to use $\CMcal{D}_m(\tilde{\lambda};\tilde{\mu})$ to represent some $m$-strip diagrams for small $m$.
\end{remark}
\begin{example}
See Figure~\ref{F:1} for an example of $6$-strip diagram with head partition $\tilde{\lambda}$ and tail partition $\tilde{\mu}$ from \cite{BR:07}. The standard diagrams of $\tilde{\lambda}=(3,1),\tilde{\mu}=(2,2)$ are rotated and attached to the body of the $6$-strip diagram.
\end{example}
\begin{center}
\begin{figure}[htbp]
\includegraphics[scale=0.7]{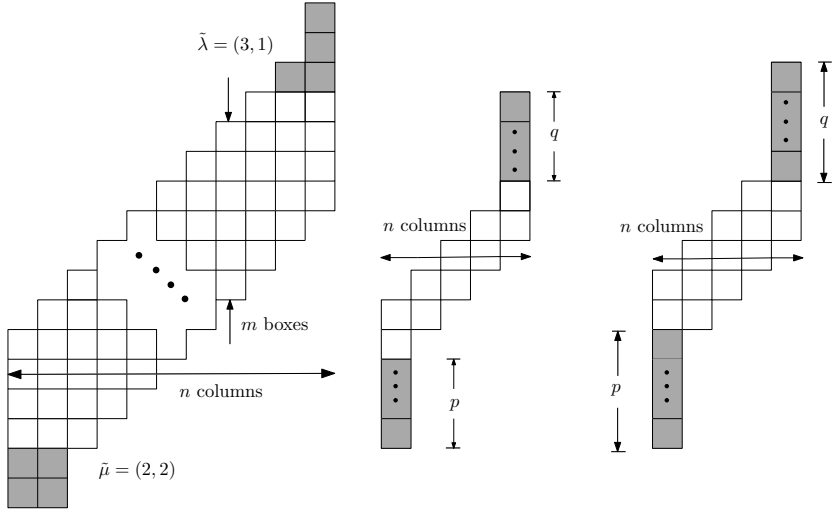}
\caption{A $6$-strip diagram $\CMcal{D}_{6}(\tilde{\lambda};\tilde{\mu})$, a $2$-strip diagram $\CMcal{D}_2((q);(p))$ and a $3$-strip diagram $\CMcal{D}_3((q);(p))$ (left, middle, right).
\label{F:1}}
\end{figure}
\end{center}
To avoid confusion, we adopt the definitions and notations of Euler numbers and tangent numbers from \cite{BR:07}. A permutation $\pi=\pi_1\pi_2\cdots \pi_n\in S_n$ is called an {\it up-down permutation} if $\pi_1<\pi_2>\pi_3<\pi_4>\cdots$. It is well-known
that the exponential generating function of the numbers $A_n$ of up-down
permutations of $[n]=\{1,2,\ldots,n\}$ is
\begin{eqnarray}\label{E:andre}
\sum_{n\ge 0}A_n\frac{x^n}{n!}=\sec x+\tan x.
\end{eqnarray}
This is also called Andr\'{e}'s theorem \cite{And:79}, which connects the numbers $A_n$ with the {\em Euler numbers} $E_n$ and {\em tangent numbers} $T_n$
by the Taylor expansions of $\sec x$ and $\tan x$, that is,
\begin{align*}
\sec x=\sum_{n=0}^{\infty}\frac{(-1)^n E_{2n}x^{2n}}{(2n)!}\quad\mbox{ and }\quad
\tan x=\sum_{n=1}^{\infty}\frac{T_{n}x^{2n-1}}{(2n-1)!}.
\end{align*}
This implies that
\begin{align}\label{E:eutan}
A_{2n}=(-1)^n E_{2n}\quad\mbox{ and }\quad A_{2n-1}=T_n.
\end{align}
It should be mentioned that Euler numbers are defined differently in some literature \cite{LP,Stanley:ec2}.

It is clear that an up-down permutation of $[2n]$ can be identified as a $2$-strip tableau of shape $\CMcal{D}_2(\varnothing;\varnothing)$. By thickening the $2$-strip diagram, Baryshnikov and Romik \cite{BR:07} introduced the $m$-strip diagram and enumerated the $m$-strip tableaux via transfer operators, which proved that the determinant to count $m$-strip tableaux has order $\lfloor m/2\rfloor$. This is certainly to their advantage that Baryshnikov and Romik's determinant for $(2k+1)$-strip diagrams is much simpler than the one directly from Hamel and Goulden's theorem (Theorem~\ref{T:HG}). We next recall the Baryshnikov and Romik's determinant for the $m$-strip tableaux. We define the numbers
\begin{align}\label{E:Abar}
\bar{A}_n=\frac{A_n}{n!},\,\tilde{A}_n=\frac{\bar{A}_n}{2^{n+1}-1}\,\mbox{ and }\,
\hat{A}_n=\frac{(2^n-1)\bar{A}_n}{2^n(2^{n+1}-1)},
\end{align}
%one will see that the number of $2k$-strip tableaux and $(2k+1)$-strip tableaux with $n$ columns and without head and tail are respectively the determinant of $\bar{A}_i$'s and the determinant of $(\tilde{A}_i+\hat{A}_i)$'s and $(\tilde{A}_i-\hat{A}_i)$'s.
and denote the head Young diagram by $\tilde{\lambda}=(\tilde{\lambda}_1,\tilde{\lambda}_2,\ldots,\tilde{\lambda}_k)$ and the tail Young diagram by $\tilde{\mu}=(\tilde{\mu}_1,\tilde{\mu}_2,\ldots,\tilde{\mu}_k)$ where $k=\lfloor m/2\rfloor$. For any non-negative integers $p,q$, we denote by $\alpha_{n,2}(p,q)$ and $\alpha_{n,3}(p,q)$ the number of $2$-strip tableaux of shape $\CMcal{D}_{2}((q);(p))$ and the number of $3$-strip tableaux of shape $\CMcal{D}_{3}((q);(p))$ where the empty partition $(0)$ is denoted by $\varnothing$. In other words,
%%%%%%%%%%%%%%%%%%%%%%%%%%%%%%%%%%
\begin{align*}
\alpha_{n,2}(p,q)=f^{\CMcal{D}_{2}((q);(p))}\quad \mbox{and} \quad\alpha_{n,3}(p,q)=f^{\CMcal{D}_{3}((q);(p))}.
\end{align*}
In particular, $\alpha_{n,2}(0,0)=A_{2n}$ and some values of $\alpha_{n,3}(p,q)$ are given in Theorem~\ref{C:345}. Note that $\alpha_{n,2}(p,q)=\alpha_{n,2}(q,p)$ holds for any non-negative integers $p$ and $q$. This is true because for any standard Young tableau $T$ of shape $\CMcal{D}_{2}((q);(p))$, if we replace every entry $w$ of $T$ by $2n+p+q+1-w$ and flip the diagram $\CMcal{D}_{2}((q);(p))$ upside-down and reverse it left-to-right, we obtain a standard Young tableau of shape $\CMcal{D}_{2}((p);(q))$. Similarly $\alpha_{n,3}(p,q)=\alpha_{n,3}(q,p)$ holds for any non-negative integers $p$ and $q$.
Furthermore, we define the numbers $X_{2n-1}(p,q)$ and $Y_{2n-1}(p,q)$ as below:
\begin{align*}
X_{2n-1}(p,q)=\frac{\alpha_{n,2}(p,q)}{(2n+p+q)!}\quad \mbox{and} \quad
Y_{2n-2}(p,q)=\frac{\alpha_{n,3}(p,q)}{(3n+p+q-2)!}.
\end{align*}
Our notation $\alpha_{n,2}(p,q)$ is $\alpha_n$ in \cite{BR:07} and we need the parameters $p,q$ to describe the thickened strips later. For the readers' convenience, we should mention that the left $2$-strip diagram in Figure $4$ of \cite{BR:07} should be the middle one in Figure~\ref{F:1}. Baryshnikov and Romik proved that
%%%%%%%%%%%%%%%%%%%%%%%%%%%%%%%%%%%
\begin{theorem}[\cite{BR:07}]\label{T:1}
Let $L_i=\tilde{\lambda}_i+k-i$ and $M_i=\tilde{\mu}_i+k-i$ for $1\le i\le k=\lfloor m/2\rfloor$. Then the number of standard Young tableaux of shape $\CMcal{D}_m(\tilde{\lambda};\tilde{\mu})$ is given by
\begin{align}\label{E:even}
f^{\CMcal{D}_m(\tilde{\lambda};\tilde{\mu})}&=(-1)^{\binom{k}{2}}\vert \CMcal{D}_m(\tilde{\lambda};\tilde{\mu})\vert!\det[X_{2n-m+1}(L_i,M_j)]_{i,j=1}^k \,\mbox{ if }\, m=2k \,\mbox{ or by}\\
\label{E:odd}f^{\CMcal{D}_m(\tilde{\lambda};\tilde{\mu})}&=(-1)^{\binom{k}{2}}\vert \CMcal{D}_m(\tilde{\lambda};\tilde{\mu})\vert!\det[Y_{2n-m+1}(L_i,M_j)]_{i,j=1}^k
\,\mbox{ if }\, m=2k+1.
\end{align}
\end{theorem}
%%%%%%%%%%%%%%%%%%%%%%%%%%%%%%%%%%%
\begin{remark}
Theorem~\ref{T:1} is a combination of Theorem 4 and 5 in \cite{BR:07}. Here we use the combinatorial interpretations of $\alpha_{n,2}(p,q),\alpha_{n,3}(p,q)$ to introduce the numbers $X_{2n-1}(p,q),Y_{2n-2}(p,q)$, whose expressions in terms of the numbers $\hat{A}_i$, $\tilde{A}_i$ and $\bar{A}_i$ can be derived by the recursions of $\alpha_{n,2}(p,q)$ and $\alpha_{n,3}(p,q)$. Here we omit the computational details.
\end{remark}
Baryshnikov and Romik \cite{BR:07} also presented some explicit formulas for small $m$.
%which are obtained by simplifying the determinants in (\ref{E:even}) and (\ref{E:odd}) for small $m$. In addition,
We will establish Theorem~\ref{C:345} by decomposing $3$-strip tableaux directly and by choosing two different outside nested decompositions respectively for $4,5$-strip tableaux.
\begin{theorem}[\cite{BR:07}]\label{C:345}
Some numbers of $3$-strip tableaux are
\begin{align}\label{E:3strip}
\alpha_{n,3}(0,0)&=f^{\CMcal{D}_3(\varnothing;\varnothing)}
=\frac{(3n-2)!T_n}{(2n-1)!2^{2n-2}}=\frac{(3n-2)!\bar{A}_{2n-1}}{2^{2n-2}},\\
\label{E:3strip1}\alpha_{n,3}(0,1)&=f^{\CMcal{D}_3((1);\varnothing)}
=\frac{(3n-1)!T_n}{(2n-1)!2^{2n-1}}=\frac{(3n-1)!\bar{A}_{2n-1}}{2^{2n-1}},\\
\label{E:3strip2}\alpha_{n,3}(1,1)&=f^{\CMcal{D}_3((1);(1))}
=\frac{(3n)!(2^{2n-1}-1)T_n}{(2n-1)!2^{2n-1}(2^{2n}-1)}
=(3n)!\hat{A}_{2n-1}.
\end{align}
Some numbers of $4$-strip tableaux are
\begin{align}\label{E:4strip1}
f^{\CMcal{D}_4(\varnothing;\varnothing)}&=\binom{4n-2}{2n-1}T_n^2
+\binom{4n-2}{2n-2}E_{2n-2}E_{2n}=(4n-2)!
\det\left[\begin{array}{ll} \bar{A}_{2n-1} & \bar{A}_{2n}\\
\bar{A}_{2n-2} & \bar{A}_{2n-1}
\end{array}\right],\\
\label{E:4strip2}f^{\CMcal{D}_4((1);(1))}&=\binom{4n}{2n}E_{2n}^2
-\binom{4n}{2n-2}E_{2n-2}E_{2n+2}=(4n)!\det\left[\begin{array}{ll} \bar{A}_{2n} & \bar{A}_{2n+2}\\
\bar{A}_{2n-2} & \bar{A}_{2n}
\end{array}\right],
\end{align}
and the number of $5$-strip tableaux without head and tail is
\begin{align}\label{E:5strip}
f^{\CMcal{D}_5(\varnothing;\varnothing)}
=\frac{(5n-6)!\,T_{n-1}^2}{((2n-3)!)^2\,2^{4n-6}(2^{2n-2}-1)}
=(5n-6)!\det\left[\begin{array}{ll} \tilde{A}_{2n-3} & \hat{A}_{2n-3}\\
\hat{A}_{2n-3} & \tilde{A}_{2n-3}
\end{array}\right].
\end{align}
\end{theorem}
\subsection{Proof of Theorem~\ref{T:1} and Theorem \ref{C:345}}
\subsubsection{Proof of (\ref{E:even})}
We count the number $f^{\CMcal{D}_{2k}(\tilde{\lambda};\tilde{\mu})}$ of $2k$-strip tableaux by choosing an outside decomposition $\phi=(\theta_1,\theta_2,\ldots,\theta_k)$ of the $2k$-strip diagram $\CMcal{D}_{2k}(\tilde{\lambda};\tilde{\mu})$, which is a special outside nested decomposition without common special corners. Given a $2k$-strip diagram $\CMcal{D}_{2k}(\tilde{\lambda};\tilde{\mu})$, we can peel this diagram off into successive maximal outer strips $\theta_1,\theta_{2},\ldots,\theta_{k}$ beginning from the outside; see the left one in Figure~\ref{F:pic-6}.
\begin{center}
\begin{figure}[htbp]
\includegraphics[scale=0.7]{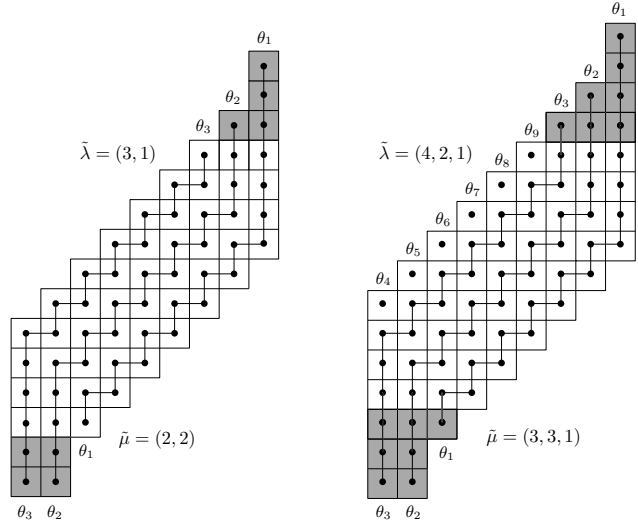}
\caption{The outside nested decomposition $\phi=(\theta_1,\theta_2,\theta_3)$ that we choose for the $6$-strip diagram $\CMcal{D}_6((3,1);(2,2))$ (left) and the outside nested decomposition that we will not choose for the $7$-strip diagram $\CMcal{D}_7((4,2,1);(3,3,1))$ (right).
\label{F:pic-6}}
\end{figure}
\end{center}
We recall the numbers $L_i=\tilde{\lambda}_i+k-i$ and $M_i=\tilde{\mu}_i+k-i$, for $1\le i\le k$. In the outside decomposition $\phi$, every strip $\theta_i$ is a $2$-strip of $(n-k+1)$ columns, with head partition $(L_i)$ and tail partition $(M_{k-i+1})$. The number of such tableaux are denoted by $\alpha_{n-k+1,2}(L_i,M_{k-i+1})$, that is, $f^{\theta_i}=\alpha_{n-k+1,2}(L_i,M_{k-i+1})$. By Definition \ref{D:operator1}, we see that the thickened cutting strip $H(\phi)$ is a $2$-strip of $(n-k+1)$ columns, with head partition $(L_1)$ and tail partition $(M_1)$. So it follows that $\theta_i\#\theta_j$ is a $2$-strip diagram with $(n-k+1)$ columns, with head partition $(L_i)$ and with tail partition $(M_{k-j+1})$. Consequently, $f^{\theta_i\#\theta_j}=\alpha_{n-k+1,2}(M_{k-j+1},L_i)=\alpha_{n-k+1,2}(L_i,M_{k-j+1})$. By Corollary~\ref{C:thick2} we know that the number $f^{\CMcal{D}_{2k}(\tilde{\lambda};\tilde{\mu})}$ of standard Young tableaux of $2k$-strip shape with $n$ columns, is expressed as a determinant where the $(i,j)$-th entry is $\alpha_{n-k+1,2}(L_i,M_{k-j+1})/(2n-2k+L_i+M_{k-j+1}+2)!
=X_{2n-2k+1}(L_i,M_{k-j+1})$. That is to say,
\begin{align*}
f^{\CMcal{D}_{2k}(\tilde{\lambda};\tilde{\mu})}&=\vert\CMcal{D}_{2k}(\tilde{\lambda};\tilde{\mu})
\vert!\det[X_{2n-2k+1}(L_i,M_{k-j+1})]_{i,j=1}^{k}\\
&=(-1)^{\binom{k}{2}}\vert\CMcal{D}_{2k}(\tilde{\lambda};\tilde{\mu})
\vert!\det[X_{2n-2k+1}(L_i,M_j)]_{i,j=1}^{k},
\end{align*}
which is (\ref{E:even}).\qed
\subsubsection{Proof of (\ref{E:odd})}\label{ss:odd}
We observe that any outside decomposition of $(2k+1)$-strip diagram will not reduce the order of the Jacobi-Trudi determinant in Theorem~\ref{T:Jacobi-Trudi} because the minimal number of strips contained in any outside decomposition is exactly the number of columns in any $(2k+1)$-strip diagram $\CMcal{D}_{2k+1}(\tilde{\lambda};\tilde{\mu})$; see the outside decomposition of the $7$-strip diagram $\CMcal{D}_7((4,2,1);(3,3,1))$ in Figure~\ref{F:pic-6}.

Given a $(2k+1)$-strip diagram $\CMcal{D}_{2k+1}(\tilde{\lambda};\tilde{\mu})$, we can peel this diagram off into successive maximal outer thickened strips $\Theta_1,\Theta_{2},\ldots,\Theta_{k}$ beginning from the outside; see Figure~\ref{F:pic-81}.
\begin{center}
\begin{figure}[htbp]
\includegraphics[scale=0.7]{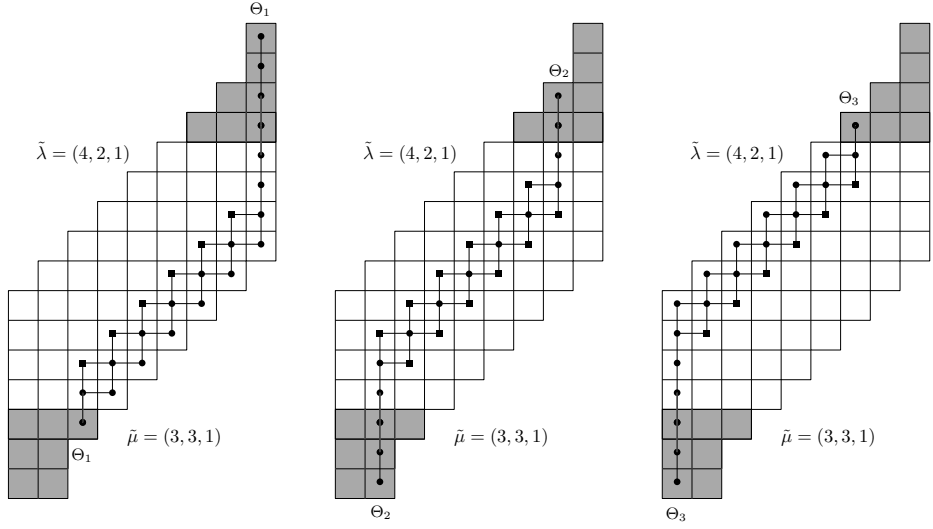}
\caption{The outside nested decomposition $\Phi=(\Theta_1,\Theta_2,\Theta_3)$ of the $7$-strip diagram $\CMcal{D}_7((4,2,1);(3,3,1))$ where each common special corner is marked by a black square in the thickened strip.
\label{F:pic-81}}
\end{figure}
\end{center}
Consider the outside nested decomposition $\Phi=(\Theta_1,\Theta_2,\ldots,\Theta_k)$, every thickened strip $\Theta_i$ is a $3$-strip of $(n-k+1)$ columns, with head partition $(L_i)$ and tail partition $(M_{k-i+1})$. The number of such tableaux are denoted by $\alpha_{n-k+1,3}(L_i,M_{k-i+1})$, that is, $f^{\Theta_i}=\alpha_{n-k+1,3}(L_i,M_{k-i+1})$. By Definition \ref{D:operator1}, we see that the thickened cutting strip $H(\Phi)$ is a $3$-strip of $(n-k+1)$ columns, with head partition $(L_1)$ and tail partition $(M_1)$. So it follows that $\Theta_i\#\Theta_j$ is a $3$-strip diagram with $(n-k+1)$ columns, with head partition $(L_i)$ and with tail partition $(M_{k-j+1})$. Consequently, $f^{\Theta_i\#\Theta_j}=\alpha_{n-k+1,3}(M_{k-j+1},L_i)=\alpha_{n-k+1,3}(L_i,M_{k-j+1})$. By Corollary~\ref{C:thick2} we know that the number $f^{\CMcal{D}_{2k+1}(\tilde{\lambda};\tilde{\mu})}$ of standard Young tableaux of $(2k+1)$-strip shape with $n$ columns, is expressed as a determinant where the $(i,j)$-th entry is $\alpha_{n-k+1,3}(L_i,M_{k-j+1})/(3n-3k+L_i+M_{k-j+1}+1)!
=Y_{2n-2k}(L_i,M_{k-j+1})$. That is to say,
\begin{align*}
f^{\CMcal{D}_{2k+1}(\tilde{\lambda};\tilde{\mu})}
&=\vert\CMcal{D}_{2k+1}(\tilde{\lambda};\tilde{\mu})
\vert!\det[Y_{2n-2k}(L_i,M_{k-j+1})]_{i,j=1}^{k}\\
&=(-1)^{\binom{k}{2}}\vert\CMcal{D}_{2k+1}(\tilde{\lambda};\tilde{\mu})
\vert!\det[Y_{2n-2k}(L_i,M_j)]_{i,j=1}^{k},
\end{align*}
which is (\ref{E:odd}).\qed
\subsubsection{Proof of (\ref{E:3strip})-(\ref{E:3strip2})}
Here we need the parameter $n$ to describe the number of columns when we decompose the $3$-strip diagrams. So we set
\[\begin{array}{cc}
\CMcal{D}_{3n-2}=\CMcal{D}_3(\varnothing;\varnothing), & \CMcal{D}_{3n-1}=\CMcal{D}_3((1);\varnothing), \\
\CMcal{D}_{3n-1}^*=\CMcal{D}_3(\varnothing;(1)),& \CMcal{D}_{3n}=\CMcal{D}_3((1);(1)).
\end{array}\]
and let $\CMcal{C}_{3n}$ denote a $3$-strip diagram which is obtained by adding a new box to the right of the topmost and rightmost box of $\CMcal{D}_3(\varnothing;(1))$; see Figure~\ref{F:3strip}. First we have two simple observations.
%%%%%%%%%%%%%%%%%%%%%%%%%%%%%%%%%%%%%
\begin{center}
\begin{figure}[htbp]
\includegraphics[scale=0.7]{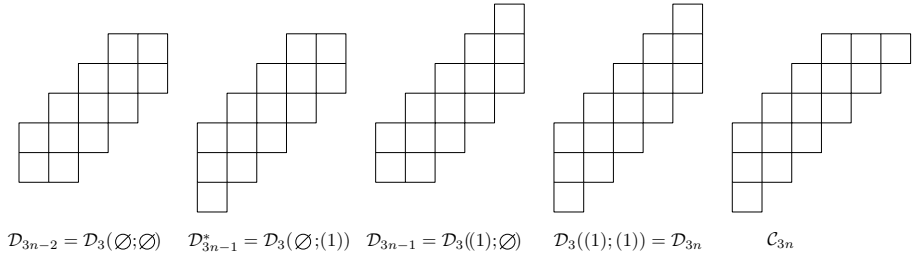}
\caption{$3$-strip diagrams
\label{F:3strip}}
\end{figure}
\end{center}
%%%%%%%%%%%%%%%%%%%%%%%%%%%%%%%%%%%%%
\begin{lemma}\label{L:funlem}
The numbers $f^{\CMcal{D}_{3n-2}}$, $f^{\CMcal{D}_{3n-1}}$ and $f^{\CMcal{D}_{3n}}$ satisfy
\begin{align}\label{E:div1}
(3n-1)f^{\CMcal{D}_{3n-2}}&=2f^{\CMcal{D}_{3n-1}},\\
\label{E:reTf}
(3n)f^{\CMcal{D}_{3n-1}}
&=f^{\CMcal{D}_{3n}}+f^{\CMcal{C}_{3n}}.
\end{align}
\end{lemma}
%%%%%%%%%%%%%%%%%%%%%%%%%%%%%%%%%%%%%%
\begin{proof}
Let $\CMcal{T}_{\sigma}$ denote the set of all standard Young tableaux of shape $\sigma$. Then, in order to prove (\ref{E:div1}), we will establish the bijection
\begin{align}\label{E:in1}
[3n-1]\times \CMcal{T}_{\CMcal{D}_{3n-2}}\rightarrow
\CMcal{T}_{\CMcal{D}_{3n-1}}\dot{\cup}
\CMcal{T}_{\CMcal{D}_{3n-1}^*}.
\end{align}
Given a pair $(T,i)$ where $i\in[3n-1]$ and $T$ is a standard Young tableau of shape $\CMcal{D}_{3n-2}$ with entries from the set $[3n-1]-\{i\}$. Suppose that the rightmost and topmost box $\alpha$ of $T$ has entry $q$. If $i<q$, then we put a box with entry $i$ on the top of box $\alpha$, which gives us a standard Young tableau of shape $\CMcal{D}_{3n-1}$ with entries from $1$ to $3n-1$. Otherwise we put a box with entry $i$ to the right of box $\alpha$, which, after transposing the rows into columns, is a standard Young tableau of shape $\CMcal{D}_{3n-1}^*$ with entries from $1$ to $3n-1$. It is clear that this procedure is reversible, so the bijection (\ref{E:in1}) follows. Furthermore, it holds that $f^{\CMcal{D}_{3i-1}}=f^{\CMcal{D}_{3i-1}^*}$ since for any standard Young tableau of shape $\CMcal{D}_{3n-1}$, if we replace every entry $q$ by $3n-q$ and flip the diagram $\CMcal{D}_{3n-1}$ upside-down and reverse it left-to-right, we obtain a standard Young tableau of shape $\CMcal{D}_{3n-1}^*$. In combination of (\ref{E:in1}), it follows that (\ref{E:div1}) is true.

In order to prove (\ref{E:reTf}), we next establish the bijection
\begin{align}\label{E:in2}
[3n]\times \CMcal{T}_{\CMcal{D}_{3n-1}^*}\rightarrow
\CMcal{T}_{\CMcal{D}_{3n}}\dot{\cup}
\CMcal{T}_{\CMcal{C}_{3n}}
\end{align}
which is analogous to (\ref{E:in1}). Given a pair $(T,i)$ where $i\in[3n]$ and $T$ is a standard Young tableau of shape $\CMcal{D}_{3n-1}^*$ with entries from the set $[3n]-\{i\}$. Suppose that the rightmost and topmost box $\alpha$ of $T$ has entry $q$. If $i<q$, then we put a box with entry $i$ on the top of box $\alpha$, which gives us a standard Young tableau of shape $\CMcal{D}_{3n}$ with entries from $1$ to $3n$. Otherwise we put a box with entry $i$ to the right of box $\alpha$, which is a standard Young tableau of shape $\CMcal{C}_{3n}$ with entries from $1$ to $3n$. This implies that (\ref{E:in2}) is a bijection, thus in view of $f^{\CMcal{D}_{3n-1}}=f^{\CMcal{D}_{3n-1}^*}$, (\ref{E:reTf}) holds.
\end{proof}
By Lemma~\ref{L:funlem} it suffices to count the numbers $f^{\CMcal{D}_{3n-2}}$ and $f^{\CMcal{C}_{3n}}$. Consider the boxes
$$(1,n-1),(2,n-2),\ldots,(n-1,1)$$
of the $3$-strip diagram $\CMcal{D}_{3n-2}$, one of these boxes has the minimal entry $1$ for any standard Young tableau from $\CMcal{T}_{\CMcal{D}_{3n-2}}$. Let $\CMcal{D}_{3n-2,i}$ be the $3$-strip diagram $\CMcal{D}_{3n-2}$ after removing the box $(i,n-i)$. Then we have
\begin{lemma}\label{L:31}
For $1\le i\le n-1$, the numbers $f^{\CMcal{D}_{3n-2,i}}$ satisfy
\begin{align}\label{E:31}
(3n-2)f^{\CMcal{D}_{3n-2,i}}&=f^{\CMcal{D}_{3n-2}}
+\binom{3n-2}{3i-1}f^{\CMcal{D}_{3i-1}}f^{\CMcal{D}_{3n-3i-1}}.
\end{align}
\end{lemma}
\begin{proof}
Let $\CMcal{S}$ denote the set of all $(3i-1)$-subsets of $[3n-2]$, we aim to construct the bijection
\begin{align}\label{E:b1}
[3n-2]\times \CMcal{T}_{\CMcal{D}_{3n-2,i}}\rightarrow
\CMcal{T}_{\CMcal{D}_{3n-2}}\dot{\cup}(\CMcal{S}
\times\CMcal{T}_{\CMcal{D}_{3i-1}^*}\times \CMcal{T}_{\CMcal{D}_{3n-3i-1}^*}),
\end{align}
from which (\ref{E:31}) follows immediately. Given a pair $(T,r)$ where $r\in[3n-2]$ and $T$ is a standard Young tableau of shape $\CMcal{D}_{3n-2,i}$ with entries from the set $[3n-2]-\{r\}$. Suppose that the entries of box $(i+1,n-i)$ and box $(i,n-i+1)$ are $q_1$ and $q_2$ in $T$, we set $q=\min\{q_1,q_2\}$. If $r<q$, then we add a box $(i,n-i)$ with entry $r$ to $T$, which is a standard Young tableau of shape $\CMcal{D}_{3n-2}$.

If $r>q=q_1$, then we consider a segment of $T$ from the starting box of $\CMcal{D}_{3n-2,i}$ to box $(i+1,n-i)$ and we add a box with entry $r$ to the right of box $(i+1,n-i)$, which, after transposing the rows into columns, leads to a standard Young tableau of shape $\CMcal{D}_{3n-3i-1}^*$ with entries coming from a $(3n-3i-1)$-subset $A$ of $[3n-2]$. Moreover, the segment of $T$ from box $(i+1,n-i+1)$ to the ending box of $\CMcal{D}_{3n-2,i}$, is a standard Young tableau of shape $\CMcal{D}_{3i-1}^*$ with entries coming from the complement set $A^c$ of $A$ with respect to $[3n-2]$.

If $r>q=q_2$, then we consider a segment of $T$ from box $(i,n-i+1)$ to the ending box of $\CMcal{D}_{3n-2,i}$ and we add a box with entry $r$ right below the box $(i,n-i+1)$, which leads to a standard Young tableau of shape $\CMcal{D}_{3i-1}^*$ with entries coming from a $(3i-1)$-subset $B$ of $[3n-2]$. Moreover, the segment of $T$ from the starting box of $\CMcal{D}_{3n-2,i}$ to box $(i+1,n-i+1)$, which, after transposing the rows into columns, is a standard Young tableau of shape $\CMcal{D}_{3n-3i-1}^*$ with entries coming from the complement set $B^c$ of $B$ with respect to $[3n-2]$.

Conversely, given a standard Young tableau $T_0$ of shape $\CMcal{D}_{3n-2}$, we set $r$ to be the entry of box $(i,n-i)$ in $T_0$ and after we remove box $(i,n-i)$ from $T_0$, we obtain a standard Young tableau of shape $\CMcal{D}_{3n-2,i}$. Given a triple $(D,T_1,T_2)$ where $T_1$ is a standard Young tableau of shape $\CMcal{D}_{3i-1}^*$ with entries from $D\in\CMcal{S}$, and $T_2$ is a standard Young tableau of shape $\CMcal{D}_{3n-3i-1}^*$ with entries from the complement set $D^c$.

Suppose that the entry of box $(i,1)$ in $T_1$ is $q_3$ and the entry of box $(n-i,1)$ in $T_2$ is $q_4$, if $q_3>q_4$, we remove the box $(n-i+1,1)$ of $T_2$, then transpose it from columns into rows and put the box with entry $q_4$ to the left of box $(i+1,1)$ of $T_1$. This gives us a standard Young tableau of shape $\CMcal{D}_{3n-2,i}$ such that the entry of box $(i,n-i+1)$ is larger than the one of box $(i+1,n-i)$ and we choose $r$ to be the entry of box $(n-i+1,1)$ of $T_2$, so that $r>q_4$.

If $q_3<q_4$, we transpose $T_2$ from columns into rows, then put its rightmost and topmost box right below the box with entry $q_3$ after we remove the box $(i+1,1)$ of $T_1$. This gives us a standard Young tableau of shape $\CMcal{D}_{3n-2,i}$ such that the entry of box $(i,n-i+1)$ is smaller than the one of box $(i+1,n-i)$ and we choose $r$ to be the entry of box $(i+1,1)$ of $T_1$, so that $r>q_3$.

Since all cases are disjoint and cover all possible scenarios, (\ref{E:b1}) is a bijection. In view of $f^{\CMcal{D}_{3i-1}}=f^{\CMcal{D}_{3i-1}^*}$, (\ref{E:31}) follows.
\end{proof}
\begin{example}
For $i=2$, we consider the pairs $(T_1,4)$ and $(T_2,6)$. Since $4<\min\{6,8\}$, we put a box with entry $4$ to $T_1$. Since $6>\min\{4,8\}$, we separate $T_2$ of shape $\CMcal{D}_{13,2}$ into two standard Young tableaux of shapes $D_{5}^*$ and $D_{8}^*$.
\end{example}
%%%%%%%%%%%%%%%%%%%%%%%%%%%%%%%%%%%%%
\begin{center}
\begin{figure}[htbp]
\includegraphics[scale=0.8]{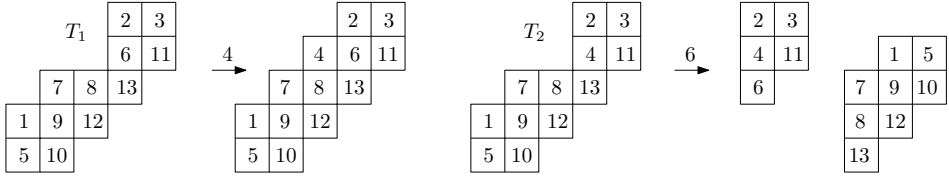}
\caption{Two examples of the bijection (\ref{E:b1}) in Lemma~\ref{L:31}.
\label{F:3strip}}
\end{figure}
\end{center}
%%%%%%%%%%%%%%%%%%%%%%%%%%%%%%%%%%%%%
Let $\CMcal{C}_{3n,i}$ be the $3$-strip diagram $\CMcal{C}_{3n}$ after removing the box $(i,n-i)$, we can decompose the skew diagram $\CMcal{C}_{3n}$ in exactly the same way. So we omit the proof of Lemma~\ref{L:32}.
\begin{lemma}\label{L:32}
For $1\le i\le n-1$, the numbers $f^{\CMcal{C}_{3n,i}}$ satisfy
\begin{align}\label{E:32}
(3n)f^{\CMcal{C}_{3n,i}}&=f^{\CMcal{C}_{3n}}
+\binom{3n}{3i}f^{\CMcal{C}_{3i}}f^{\CMcal{C}_{3n-3i}}.
\end{align}
\end{lemma}
With the help of recursions (\ref{E:31}) and (\ref{E:32}), we could use the generating function approach to finally derive the numbers of $3$-strip tableaux.

{\em Proof of (\ref{E:3strip})-(\ref{E:3strip2})}. Summing (\ref{E:31}) and (\ref{E:32}) over all $i$ gives us
\begin{align}\label{E:rec1}
(2n-1)f^{\CMcal{D}_{3n-2}}&=\sum_{i=1}^{n-1}
\binom{3n-2}{3i-1}f^{\CMcal{D}_{3i-1}}
f^{\CMcal{D}_{3n-3i-1}}\\
\label{E:rec2}
(2n+1)f^{\CMcal{C}_{3n}}&=\sum_{i=1}^{n-1}\binom{3n}{3i}
f^{\CMcal{C}_{3i}}f^{\CMcal{C}_{3n-3i}}.
\end{align}
We can translate the recursions (\ref{E:rec1}) and (\ref{E:rec2}) into two identities of exponential generating functions. We define that
\begin{eqnarray*}
f(x)=\sum_{n\ge 1}\frac{f^{\CMcal{D}_{3n-2}}}{(3n-2)!}x^{2n-1},\,\,
g(x)=\sum_{n\ge 1}\frac{f^{\CMcal{D}_{3n-1}}}{(3n-1)!}x^{2n-1},\,\,
h(x)=\sum_{n\ge 1}\frac{f^{\CMcal{C}_{3n}}}{(3n)!}x^{2n-1}.
\end{eqnarray*}
From (\ref{E:div1}) we have $f(x)=2g(x)$. Furthermore, (\ref{E:rec1}) is
equivalent to $f'(x)=1+g(x)^2$ where $f(0)=0$. This leads to a unique solution,
$g(x)=\tan(x/2)$. Together with the exponential generating function for $A_{2n-1}$; see (\ref{E:andre}), we can prove (\ref{E:3strip}) and (\ref{E:3strip1}).
Similarly, (\ref{E:rec2}) is equivalent to $-h'(x)=-1+\frac{2}{x}h(x)-h^2(x)$ where $h(0)=1$. This yields a unique solution $$h(x)=-\frac{1}{\tan x}+\frac{1}{x}=\frac{1}{3}x+\frac{1}{45}x^3+\cdots,$$
from which we can derive the numbers $f^{\CMcal{C}_{3n}}$ by expanding $h(x)$, i.e.,
\begin{align}\label{E:c}
f^{\CMcal{C}_{3n}}=\frac{(3n)!A_{2n-1}}{(2n-1)!
(2^{2n}-1)}=(3n)!\tilde{A}_{2n-1}.
\end{align}
Together with (\ref{E:reTf}) and (\ref{E:3strip1}), %the exponential generating function for $f^{\CMcal{D}_{3n}}$ is given by
%\begin{align*}
%\sum_{n\ge 0}\frac{f^{\CMcal{D}_{3n}}}{(3n)!}x^{2n}=\frac{x}{\sin(x)}.
%\end{align*}
%By considering the expansion of $x/\sin(x)$,
(\ref{E:3strip2}) is proved.\qed
\subsubsection{Proof of (\ref{E:4strip1})-(\ref{E:4strip2})}
For the $4$-strip diagram $\CMcal{D}_4(\varnothing;\varnothing)$, we choose another outside decomposition $\phi^*=(\theta_{1}^*,\theta_{2}^*,\ldots,\theta_k^*)$, which is slightly different to the one $\phi=(\theta_1,\theta_2,\ldots,\theta_k)$ for the $2k$-strip diagram $\CMcal{D}_{2k}(\tilde{\lambda};\tilde{\mu})$. The benefit to make such a slight change is that the determinant in (\ref{E:even}) is further simplified, which only has the numbers $\bar{A}_i$ as entries.

We call a $2$-strip {\em a zig-zag strip} if the corresponding standard Young tableaux are in bijection with up-down permutations. For instance, the $2$-strip diagram $\CMcal{D}_2(\varnothing;\varnothing)$ is a zig-zag strip and all strips in Figure~\ref{F:pic-82} are zig-zag strips.
\begin{center}
\begin{figure}[htbp]
\includegraphics[scale=0.7]{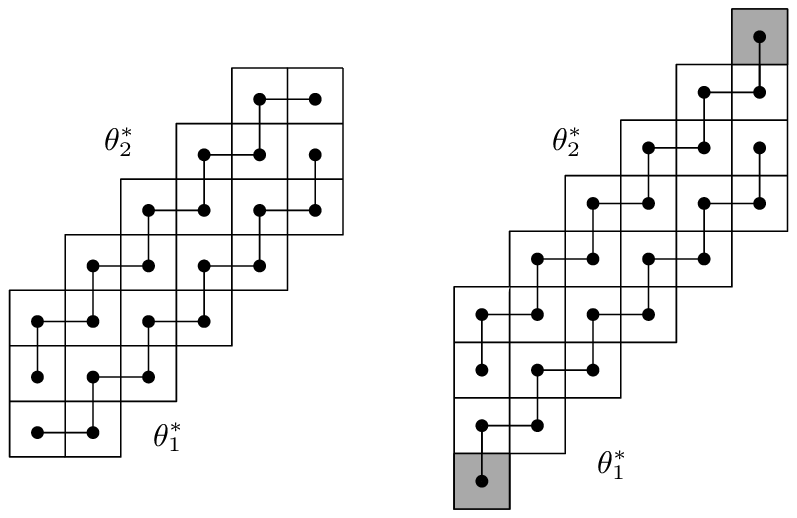}
\caption{The outside nested decompositions $\phi^*=(\theta_1^*,\theta_2^*)$ of the $4$-strip diagrams $\CMcal{D}_4(\varnothing;\varnothing)$ (left) and $\CMcal{D}_4((1);(1))$ (right).
\label{F:pic-82}}
\end{figure}
\end{center}
For the $4$-strip diagrams $\CMcal{D}_4(\varnothing;\varnothing)$, we can peel each diagram off into successive maximal zig-zag outer strips $\theta_1^*,\theta_2^*,\ldots,\theta_{k}^*$ beginning from the outside. See the left one in Figure~\ref{F:pic-82}. It is clear that the numbers of any zig-zag strip $\theta_1^*$ or $\theta_2^*$ are Euler numbers or tangent numbers. By Definition~\ref{D:operator1} we find that
\begin{align*}
f^{\theta_1^*}=f^{\theta_2^*}=A_{2n-1},\quad f^{\theta_1^*\#\theta_2^*}=A_{2n-2}\quad\mbox{ and }\quad
f^{\theta_2^*\#\theta_1^*}=A_{2n}.
\end{align*}
By Corollary~\ref{C:thick2}, we can prove (\ref{E:4strip1}) and (\ref{E:4strip2}) follows in the same way.\qed
\subsubsection{Proof of (\ref{E:5strip})}
For the $5$-strip diagram $\CMcal{D}_5(\varnothing;\varnothing)$, we choose another outside nested decomposition $\Phi^*=(\Theta_{1}^*,\Theta_{2}^*,\ldots,\Theta_k^*)$, which is slightly different to the one $\Phi=(\Theta_1,\Theta_2,\ldots,\Theta_k)$ for the $(2k+1)$-strip diagram $\CMcal{D}_{2k+1}(\tilde{\lambda};\tilde{\mu})$. The benefit to make such a slight change is that the determinant in (\ref{E:odd}) is further simplified, which only has the numbers $\hat{A}_i$ and $\tilde{A}_j$ as entries.

We call a $3$-strip {\em a zig-zag thickened strip} if the number of such $3$-strip tableaux is one of the numbers $f^{\CMcal{D}_{3n-2}}$, $f^{\CMcal{D}_{3n-1}}$, $f^{\CMcal{D}_{3n}}$ and $f^{\CMcal{C}_{3n}}$.
For instance, two thickened strips in Figure~\ref{F:pic-83} are zig-zag thickened strips. It is clear that the numbers of the zig-zag thickened strips $\Theta_1^*$ or $\Theta_2^*$ are $f^{\CMcal{C}_{3n-3}}$. By Definition~\ref{D:operator1} we find that
\begin{align*}
f^{\Theta_1^*}=f^{\Theta_2^*}=f^{\CMcal{C}_{3n-3}}\quad \mbox{ and }\quad f^{\Theta_1^*\#\Theta_2^*}=f^{\Theta_2^*\#\Theta_1^*}=f^{\CMcal{D}_{3n-3}}
=f^{\CMcal{D}_3((1);(1))}.
\end{align*}
By Corollary~\ref{C:thick2}, we can prove that
\begin{align*}
f^{\CMcal{D}_5(\varnothing;\varnothing)}=\frac{(5n-6)!}{(3n-3)!^2}
((f^{\CMcal{C}_{3n-3}})^2-(f^{\CMcal{D}_{3}((1);(1))})^2).
\end{align*}
Combining (\ref{E:3strip2}) and (\ref{E:c}), we can conclude that (\ref{E:5strip}) is true.\qed
%%%%%%%%%%%%%%%%%%%%%%%%%%%%%%%%%%%%%%%%%%%%%%
\begin{center}
\begin{figure}[htbp]
\includegraphics[scale=0.7]{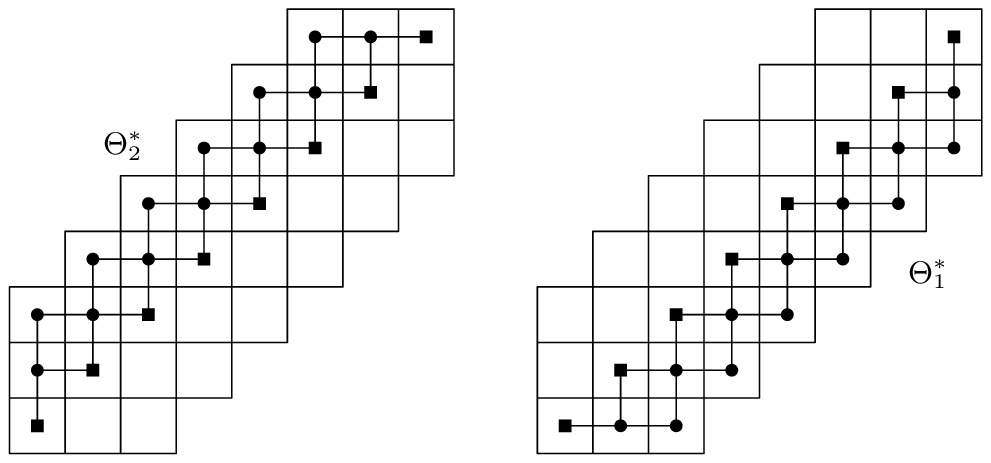}
\caption{The outside nested decomposition $\Phi^*=(\Theta_1^*,\Theta_2^*)$ of the $5$-strip diagram $\CMcal{D}_5(\varnothing;\varnothing)$.
\label{F:pic-83}}
\end{figure}
\end{center}
%%%%%%%%%%%%%%%%%%%%%%%%%%%%%%%%%%%%%%%%%%%%%%

\section*{Acknowledgement}
The author would like to thank Ang\`{e}le Hamel, Igor Pak, Dan Romik and John Stembridge for their very helpful suggestions and encouragements and would like to give special thanks to the joint seminar Arbeitsgemeinschaft Diskrete Mathematik for their valuable feedback. This work was partially done during my stay in the AG Algorithm and Complexity, Technische Universit\"{a}t Kaiserslautern, Germany. The author also thanks Markus Nebel, Sebastian Wild and Raphael Reitzig for their kind help and support.

\end{document}